\documentclass[12pt]{amsart}

\usepackage{amsmath,amssymb,amscd,amsxtra,upref,stmaryrd}
\usepackage{latexsym}
\usepackage[dvips]{graphics,epsfig}
\usepackage{enumerate,mathrsfs}
\usepackage[colorlinks=true, pdfstartview=FitV, linkcolor=blue, citecolor=blue, urlcolor=blue]{hyperref}

\allowdisplaybreaks

\numberwithin{equation}{section}

\newtheorem{theorem}{Theorem}[section]
\newtheorem{lemma}[theorem]{Lemma}
\newtheorem{proposition}[theorem]{Proposition}
\newtheorem{corollary}[theorem]{Corollary}

\newtheorem{thm}{Theorem} 

\newtheorem{lem}[thm]{Lemma} 
\newtheorem{prop}[thm]{Proposition} 

\theoremstyle{remark}

\newtheorem{remark}{Remark}[section]

\newcommand{\na}{\mathbb{N}}
\newcommand{\re}{\mathbb{R}}
\newcommand{\ent}{\mathbb{Z}}

\newcommand{\abs}[1]{\left\vert #1 \right\vert}
\newcommand{\fr}[2]{{\textstyle \frac{#1}{#2}}}

\newcommand{\norm}[2]{\left\|#1\right\|_{#2}}

\newcommand{\cz}{Calder\'on-Zygmund }
\newcommand{\rn}{{{\mathbb R}^n}}
\newcommand{\rtn}{\re^{2n}}

\newcommand{\zz}{\mathbb Z}

\newcommand{\f}{\varphi}

\newcommand{\hf}{\widehat{f}}
\newcommand{\hg}{\widehat{g}}

\newcommand{\supp}{\mathrm{supp}}
\newcommand{\sw}{{\mathcal{S}}(\rn)}
\newcommand{\swp}{{\mathcal{S}'}(\rn)}

\newcommand{\F}{\mathcal{F}}

\newcommand{\B}{\mathscr{B}}
\newcommand{\Q}{\mathcal{Q}}

\newcommand{\Ff}{\mathscr{F}}

\newcommand{\eixxe}{e^{2\pi i x \cdot (\xi + \eta)}}

\newcommand{\dxe}{\,d\xi d\eta}

\newcommand{\beq}{\begin{equation}}
\newcommand{\eeq}{\end{equation}}

\newcommand{\bal}{\begin{align*}}
\newcommand{\eal}{\end{align*}}

\newcommand{\subRn}{{{\mathbb R}^n}}

\DeclareMathOperator*{\essinf}{ess\,inf}
\DeclareMathOperator*{\esssup}{ess\,sup}

\newcommand{\pp}{{p(\cdot)}}
\newcommand{\cpp}{{p'(\cdot)}}
\newcommand{\Lp}{L^{p(\cdot)}}
\newcommand{\Pp}{\mathcal P}
\newcommand{\qq}{{q(\cdot)}}

\newcommand{\rr}{{r(\cdot)}}

\newcommand{\bp}{{\bar{p}(\cdot)}}
\newcommand{\cbp}{{\bar{p}'(\cdot)}}
\newcommand{\bq}{{\bar{q}(\cdot)}}
\newcommand{\cbq}{{\bar{q}'(\cdot)}}
\newcommand{\br}{{\bar{r}(\cdot)}}
\newcommand{\cbr}{{\bar{r}'(\cdot)}}

\newcommand{\Rf}{\mathcal{R}}

\def\Xint#1{\mathchoice
   {\XXint\displaystyle\textstyle{#1}}%
   {\XXint\textstyle\scriptstyle{#1}}%
   {\XXint\scriptstyle\scriptscriptstyle{#1}}%
   {\XXint\scriptscriptstyle\scriptscriptstyle{#1}}%
   \!\int}
\def\XXint#1#2#3{{\setbox0=\hbox{$#1{#2#3}{\int}$}
     \vcenter{\hbox{$#2#3$}}\kern-.5\wd0}}

\def\avgint{\Xint-}

\begin{document}

\title[Kato-Ponce inequalities]{Kato-Ponce inequalities on weighted and  variable Lebesgue spaces}
\author{David Cruz-Uribe, OFS \and Virginia Naibo}

\address{David Cruz-Uribe, OFS, Department of Mathematics, University of
  Alabama, Tusca\-loosa, AL 35487-0350, USA.} 
\email{dcruzuribe@ua.edu}

\address{Virginia Naibo, Department of Mathematics, Kansas State University.
138 Cardwell Hall, 1228 N. 17th Street, Manhattan, KS  66506, USA.}
\email{vnaibo@math.ksu.edu}

\thanks{The first author is supported by the NSF under grant DMS
  1362425, and by research funds from the Dean of the College of Arts
  \& Sciences. The second author is supported by the NSF under grant DMS 1500381.}


\date{\today}


\begin{abstract}
We prove  fractional Leibniz rules and related commutator estimates
  in the settings of weighted and variable Lebesgue spaces.  Our main
  tools are uniform
  weighted estimates for sequences of square-function-type operators
  and a bilinear extrapolation theorem.  We also give applications of the extrapolation theorem to the
  boundedness on variable Lebesgue spaces of certain bilinear
  multiplier operators and singular integrals.
\end{abstract}

\maketitle


\section{Introduction and main results}

For $s \geq 0$, the inhomogeneous and homogeneous $s$-th differentiation operators $J^s$ and $D^s,$ respectively, are defined via the Fourier transform as
$$
\widehat{J^s(f)}(\xi)= (1+|\xi|^2)^\frac{s}{2} \hf(\xi) \quad \text{and}\quad \widehat{D^s(f)}(\xi)= |\xi|^s \hf(\xi).
$$
The following inequalities,  known as Kato-Ponce inequalities or fractional Leibniz rules, hold for such operators and for $f,g\in\sw$:
\begin{align}
  \norm{D^s(fg)}{L^r} &\lesssim \left( \norm{D^sf}{L^{p_1}} \norm{g}{L^{q_1}}+ \norm{f}{L^{p_2}} \norm{D^sg}{L^{q_2}}\right),\label{eq:KP1}\\
   \norm{J^s(fg)}{L^r} &\lesssim \left( \norm{J^sf}{L^{p_1}} \norm{g}{L^{q_1}}+ \norm{f}{L^{p_2}} \norm{J^sg}{L^{q_2}}\right), \label{eq:KP2}
\end{align}
where $1<p_1,p_2,q_1,q_2\le \infty,$ $\frac{1}{2}<r\le \infty,$
$\frac{1}{r}=\frac{1}{p_1}+\frac{1}{q_1}=\frac{1}{p_2}+\frac{1}{q_2},$
$s>\max\{0,n(\frac{1}{r}-1)\}$ or $s$ is a non-negative even integer,
and the implicit constants depend only on the parameters involved.
The cases $1<r<\infty,$ $1<q_1,p_2\le \infty,$ $1<p_1,q_2< \infty$ and
$s>0$ for such estimates have been known for a long time and go back
to the pioneering work in \cite{MR1124294, MR951744, MR1211741} for
the study of the Euler, Navier-Stokes and Korteweg-de~Vries equations.
Very recently in \cite{MR3200091} a different
approach was used to extend~\eqref{eq:KP1} and~\eqref{eq:KP2}
to the range
$\frac{1}{2}<r\le 1$ with $s>\max\{0,n(\frac{1}{r}-1)\}$ or $s$ a
non-negative even integer.  (Also see \cite{MR3192309} for
\eqref{eq:KP2} with $\frac{1}{2}<r\le 1$ and $s>n,$ and
\cite{MR3052499} for \eqref{eq:KP1}.) The case $r=\infty$ was settled
in \cite{MR3263081} (see also \cite{MR3189525}).

In \cite{MR951744}, an
important tool in the study of the Cauchy problem for the Euler and
Navier-Stokes equations in the setting of $L^p$-based Sobolev spaces
was the following commutator estimate closely related to
\eqref{eq:KP2}:
\begin{equation*}
 \norm{J^s(fg)-fJ^s(g)}{L^p} \lesssim  \norm{J^sf}{L^p} \norm{g}{L^\infty}+ \norm{\nabla f}{L^\infty} \norm{J^{s-1}g}{L^p},
 \end{equation*}
 where $1<p<\infty$ and $s\ge 0.$   For other commutator estimates of
 the sort, see~\cite{MR3263081} and the references it contains.

\bigskip

The goal of this paper is to prove Kato-Ponce inequalities in the
settings of weighted Lebesgue spaces and variable Lebesgue spaces.
The techniques employed in the context of weighted Lebesgue spaces
also let us obtain fractional Leibniz rules in weighted Lorentz
spaces and Morrey spaces.  To state our two main results, let
$\sw$ denote the Schwartz class of rapidly decreasing functions
defined on $\rn$ and let $A_p$ denote the Muckenhoupt class of
weights.  For brevity we will use the notation $A\lesssim B$ to mean
that $A\le cB$, where $c$ is a constant that may only depend on some
of the parameters and weights used but not on the functions
involved. For the notation used in the statement of
Theorem~\ref{thm:KPweighted}, we refer the reader to Section~\ref{sec:weights}.

\begin{theorem}\label{thm:KPweighted}
Let $1< p,q<\infty$ and $\frac{1}{2}<r<\infty$ be such that $\frac{1}{r}=\frac{1}{p}+\frac{1}{q}.$
If $v\in A_{p},$  $w\in A_{q}$,  and  $s>\max\{0,n(\frac{1}{r}-1)\}$ or
$s$ is a non-negative even integer, then
for all $f,g\in \sw$, 
\begin{gather}
\norm{D^s(fg)}{L^r(v^{\frac{r}{p}}w^{\frac{r}{q}})} \lesssim  \norm{D^sf}{L^p(v)} \norm{g}{L^q(w)}+ \norm{f}{L^p(v)} \norm{D^sg}{L^q(w)},
  \label{eq:KPweighted1}\\  
\norm{J^s(fg)}{L^r(v^{\frac{r}{p}}w^{\frac{r}{q}})} \lesssim
\norm{J^sf}{L^p(v)} \norm{g}{L^q(w)}+ \norm{f}{L^p(v)}
\norm{J^sg}{L^q(w)}, \label{eq:KPweighted2}
\end{gather}
and
\begin{multline} 
\norm{D^s(fg)-fD^s(g)}{L^r(v^{\frac{r}{p}}w^{\frac{r}{q}})} \\ 
\lesssim
\norm{D^sf}{L^p(v)} \norm{g}{L^q(w)}+ \norm{\nabla f}{L^p(v)}
\norm{D^{s-1}g}{L^q(w)}, 
  \label{eq:Commweighted1}
\end{multline}
\begin{multline}
\norm{J^s(fg)-fJ^s(g)}{L^r(v^{\frac{r}{p}}w^{\frac{r}{q}})} \\
\lesssim
\norm{J^sf}{L^p(v)} \norm{g}{L^q(w)}+ \norm{\nabla f}{L^p(v)}
\norm{J^{s-1}g}{L^q(w)}. \label{eq:Commweighted2} 
\end{multline}
The implicit constants depend on $p$, $q$, $s$, $[v]_{A_p}$ and $[w]_{A_q}$.
\end{theorem}

The ``factored'' bilinear weights in Theorem~\ref{thm:KPweighted} were
introduced in~\cite{MR1947875} and further studied
in~\cite{MR2030573}.   In~\cite{MR2483720} a more general case of
bilinear $A_p$ weights was introduced.  We do not know if our result
can be proved for these weights.  However, the factored
weights are sufficient to apply extrapolation and prove our second
main result.  For the notation used in the statement of
Theorem~\ref{thm:KPvariable}, we refer the reader to Section~\ref{sec:variable}.

\begin{theorem}\label{thm:KPvariable}
  Let $p(\cdot),q(\cdot),\rr\in \Pp_0$ be exponent functions such that
  $\frac{1}{\rr}=\frac{1}{p(\cdot)}+\frac{1}{q(\cdot)}$.  Suppose
  further that
  there exist $1<p<{p}_{-}$ and $1<q<{q}_{-}$ such that
  $\left(\frac{p(\cdot)}{p}\right)',\,\left(\frac{q(\cdot)}{q}\right)'\in\B.$
  If $s>\max\{0,n(\frac{1}{r}-1)\},$ where
  $\frac{1}{r}=\frac{1}{p}+\frac{1}{q},$ or $s$ is a non-negative even
  integer, then for all $f,g\in \sw$,
\begin{align}
  &\norm{D^s(fg)}{L^{r(\cdot)}} \lesssim  \norm{D^sf}{L^{p(\cdot)}} \norm{g}{L^{q(\cdot)}}+ \norm{f}{L^{p(\cdot)}} \norm{D^sg}{L^{q(\cdot)}},
  \label{eq:KPvariable1}\\
 &\norm{J^s(fg)}{L^{r(\cdot)}} \lesssim  \norm{J^sf}{L^{p(\cdot)}} \norm{g}{L^{q(\cdot)}}+ \norm{f}{L^{p(\cdot)}} \norm{J^sg}{L^{q(\cdot)}}, \label{eq:KPvariable2}\\
 &\norm{D^s(fg)-fD^s(g)}{L^{r(\cdot)}} \lesssim  \norm{D^sf}{L^{p(\cdot)}} \norm{g}{L^{q(\cdot)}}+ \norm{\nabla f}{L^{p(\cdot)}} \norm{D^{s-1}g}{L^{q(\cdot)}},  \label{eq:Commvariable1}\\
 &\norm{J^s(fg)-fJ^s(g)}{L^{r(\cdot)}} \lesssim  \norm{J^sf}{L^{p(\cdot)}} \norm{g}{L^{q(\cdot)}}+ \norm{\nabla f}{L^{p(\cdot)}} \norm{J^{s-1}g}{L^{q(\cdot)}}. \label{eq:Commvariable2}
\end{align}
\end{theorem}

\begin{remark}
Implicit in the proof of Theorem~\ref{thm:KPweighted} is the fact that, in the case $w=v,$ it is possible to
  have different pairs $p,\,q$ for each term on the righthand side of \eqref{eq:KPweighted1}, \eqref{eq:KPweighted2}, \eqref{eq:Commweighted1} and
  \eqref{eq:Commweighted2}. Also, the  proof of Theorem~\ref{thm:KPvariable} allows for different pairs of  exponents $p(\cdot),q(\cdot)$ for each  term on the righthand side of the inequalities for variable Lebesgue spaces. 

\end{remark}

\begin{remark}
  As we will explain in Section~\ref{sec:variable}, the hypotheses of
  Theorem~\ref{thm:KPvariable} imply that $\pp$ and $\qq$ are both
  bounded, and this restriction is intrinsic to our proof.  However,
  in the scale of variable Lebesgue spaces it is possible to have
  unbounded exponents or even exponents that are equal to infinity on
  sets of positive measure.  We conjecture that there is a
  version of Theorem~\ref{thm:KPvariable} that allows unbounded
  exponents and that includes
  the endpoint inequality in~\cite{MR3263081} as a special case.
\end{remark}

\medskip

Our proofs of \eqref{eq:KPweighted1} and \eqref{eq:KPweighted2} in
Theorem~\ref{thm:KPweighted} exploit the ideas used in
\cite{MR3200091} to prove their unweighted counterparts \eqref{eq:KP1}
and \eqref{eq:KP2}.  This approach requires us to use weighted
estimates for a certain family of square-functions; moreover we need
to have good control on their norms.  The central step in this
argument is Theorem~\ref{thm:sqweighted} where we obtain uniform
estimates for such operators. We note that
Theorem~\ref{thm:sqweighted} improves the estimate gotten
in~\cite[Corollary 1]{MR3200091}; their proof relies on a weak $(1,1)$
estimate and interpolation.  This is enough for the proof of the
unweighted estimates \eqref{eq:KP1} and \eqref{eq:KP2}, but not for
their weighted analogs of Theorem~\ref{thm:KPweighted}.  We instead
prove the necessary weighted strong type estimates directly.  A novel
feature of our approach is that we avoid using the maximal operator to
estimate convolution operators and instead use an argument based on
averaging operators that was developed in a different context
in~\cite{CUMR}.  We do not know if the weak type estimates in weighted
or variable Lebesgue spaces that correspond to the weak-type estimates
proved in~\cite{MR3200091} hold.

The proof
of the weighted commutator estimates \eqref{eq:Commweighted1} and
\eqref{eq:Commweighted2} in Theorem~\ref{thm:KPweighted} relies on a
decomposition of the operators $D^s(fg)$ and $J^s(fg)$ given in
\cite{MR951744} for the case of $J^s$ as well as weighted estimates for
certain bilinear operators from~\cite{MR2030573,MR1947875}.
  
The variable Lebesgue space estimates in 
Theorem~\ref{thm:KPvariable} are a consequence of
Theorem~\ref{thm:KPweighted} and a bilinear extrapolation theorem
(Theorem~\ref{thm:extrapolation}) that allows us to obtain bilinear
estimates in variable Lebesgue spaces from bilinear estimates in
weighted Lebesgue spaces.   This result generalizes both
the extrapolation theorem for variable Lebesgue spaces
in~\cite{MR2210118} (see also \cite{MR3026953}) and the bilinear
extrapolation theorem in~\cite{MR2030573}.   Our result is interesting
in its own right as it lets us easily prove a number of other bilinear
estimates in the variable Lebesgue space setting.  To illustrate this, 
we use 
Theorem~\ref{thm:extrapolation} to prove estimates in variable Lebesgue spaces  for
several types of bilinear
operators, including bilinear Calder\'on-Zygmund operators, bilinear
multiplier operators with symbols that have limited smoothness, and
certain rough bilinear singular integrals.
 
 \medskip
 
The remainder of this paper is organized as follows.  In
Section~\ref{sec:weights} we present  definitions and basic results
about weights and weighted norm inequalities, and then state and prove
 Theorem~\ref{thm:sqweighted}.    In Section~\ref{sec:variable} we
 give the necessary definitions and background information on variable
 Lebesgue spaces, 
 and prove Theorem~\ref{thm:extrapolation}.  In
 Section~\ref{sec:applextrap} we state and prove the applications of
 Theorem~\ref{thm:extrapolation} to other bilinear operators. 
 In Section~\ref{sec:proofmain} we prove
 Theorems~\ref{thm:KPweighted} and \ref{thm:KPvariable}. Finally, in Section~\ref{sec:LM} we present  Kato-Ponce inequalites in the settings of weighted Lorentz spaces and Morrey spaces and explain how they follow from tools developed for the proof of Theorem~\ref{thm:KPweighted}.

Throughout this paper, all notation will be standard or will be defined
as needed.  We will let $\ent$ be
the set of integers, $\na$ the set of natural numbers, and $\na_0=\na\cup\{0\}$.   Unless explicitly indicated
 otherwise, all function spaces that appear will consist of 
 complex-valued functions defined on $\rn$.

{\bf Acknowledgements:}  The authors thank the referee for the careful reading of the manuscript. 

 \section{Weighted square function estimates}\label{sec:weights}
 
 The main result in this section is Theorem~\ref{thm:sqweighted},
 which gives a uniform weighted estimate for certain families of
 square-functions.  This estimate is in turn used in the proofs of our
 main results.  We divide this section into three parts.  In the first
 we give some definitions and state some known results.  In the second
 we state and prove Theorem~\ref{thm:sqweighted}.  The proof appears
 relatively short, but it depends on several technical propositions
 which we prove in the final part.

\subsection*{Preliminary results} 
By a weight $w$ we mean a non-negative, locally integrable
function defined on $\rn$.  Given $0<p<\infty,$ let $L^p(w)$ denote the class
of complex-valued, measurable functions
$f$ defined on $\rn$ such that
\[
\norm{f}{L^p(w)}=\left(\int_{\rn}\abs{f(x)}^pw(x)\,dx\right)^{\frac{1}{p}}<\infty. \] 
For brevity, when $w=1$ we write $L^p$ and $\norm{f}{L^p}$ instead of $L^p(w)$ and
$\norm{f}{L^{p}(w)}$.
 
Given a locally integrable function $f$ on $\rn$ and a set
$E\subset \rn$ of positive measure, define
\[ \avgint_E f \,dy=\frac{1}{|E|}\int_E f(y)\,dy. \]
The Hardy-Littlewood maximal operator $\mathcal{M}$ is defined as
follows: for each $x\in \rn$, let
 \[
 \mathcal{M}(f)(x)=\sup_{B} \avgint_B\abs{f(y)}\,dy \, \chi_B(x), 
 \]
 where the supremum is taken over all Euclidean balls in $\rn$.  
 
Given $1<p<\infty,$  the Muckenhoupt class $A_p$  consists of all weights $w$ such that
 \[
 [w]_p=\sup_{Q}\left(\avgint_Q w(y)\,dy\right) \left(\avgint_Q w^{-\frac{1}{p-1}}\right)^{p-1}<\infty,
 \]
 where the supremum is taken over all cubes $Q\subset \rn.$  The class
 $A_1$ consists of all weights  $w$ such that $ \mathcal{M}(w)(x)\lesssim w(x)$ for almost every $x\in\rn;$ we define 
 \[
 [w]_{A_1}=\sup_{Q}\left(\avgint_Q w(y)\,dy\right)\norm{w^{-1}}{L^\infty(Q)},
 \]
 where the supremum is taken over all cubes $Q\subset \rn.$   Note
 that $A_1\subset A_p$ for $1<p<\infty$ and $[w]_{A_p}\leq
 [w]_{A_1}$.  
 
 It is well known that for $1<p<\infty$, the Muckenhoupt condition
 characterizes the weights $w$ such that $\mathcal{M}$ is bounded on
 $L^p(w)$ (e.g., see~\cite{MR1800316}).   Below, we will need the vector-valued version of this result,
 also referred to as the weighted Fefferman-Stein inequality.  For a
 proof, see~\cite{MR604351,MR2797562}.

\begin{lem} \label{lemma:fefferman-stein}
 Given $1<p,q<\infty$ and $w\in A_p$, for all sequences $\{f_k\}$ of locally
 integrable functions defined on $\rn$,
 \begin{equation}\label{eq:FSweighted}
\bigg\|\bigg(\sum_{k\in\ent}\abs{\mathcal{M}(f_k)}^q\bigg)^{\frac{1}{q}}\bigg\|_{L^p(w)}
\lesssim  
\bigg\|\bigg(\sum_{k\in\ent}\abs{f_k}^q\bigg)^{\frac{1}{q}} \bigg\|_{L^p(w)},
\end{equation}
where the implicit constant depends on $p$, $q$, and $[w]_{A_p}$.
\end{lem}

The next result we recall is a weighted Littlewood-Paley estimate.  To
state it, let $\widehat{f}$ denote the Fourier transform of a tempered
distribution on $\rn$:  more precisely, for all $\xi\in\rn$ and
$f\in\sw$, 
\[
\widehat{f}(\xi)=\int_{\rn}f(x)\,e^{-2\pi i x\cdot \xi}\,dx.
\] 

\begin{lem} \label{lemma:square-estimate} Let $\varphi\in\sw$ be
  such that
  $\supp(\widehat{\varphi}) \subset \{\xi\in\rn: c_1<\abs{\xi}<c_2\}$
  for some $0<c_1<c_2<\infty$.  Set 
  $\varphi_k(x)=2^{kn}\varphi(2^k x)$  and suppose further that
for some  constant $c_\varphi>0$ and for all $\xi\neq 0$,
\[ \sum_{k\in\ent}\abs{\widehat{\varphi}(2^{-k}\xi)}^2=c_\varphi. \]
Then for every $1<p<\infty$, $w\in A_p$, and $f\in L^p(w)$,
\begin{equation}\label{eq:squareestimate}
\bigg\|\bigg(\sum_{k\in\ent}\abs{\varphi_k\ast
  f}^2\bigg)^{\frac{1}{2}}\bigg\|_{L^p(w)}\sim \norm{f}{L^p(w)}, 
\end{equation}
where the implicit constants depend on $\varphi$ and
$[w]_{A_p}$ but do not depend on $f$. 
\end{lem}

For a proof, see \cite[Proposition 1.9]{MR1821243} and the comment that follows; see also \cite{MR561835}.
 
\subsection*{Statement and proof of Theorem~\ref{thm:sqweighted}}
In Lemma~\ref{lemma:square-estimate} the implicit constant in
\eqref{eq:squareestimate} depends on the function $\varphi$; in
particular, if we replace $\varphi$ by a translation
$\varphi(\cdot+z)$, $z\in\rn$, then we do not know {\em a priori}
whether the constants will depend on $z$.    Our main result in this
section shows that in this case they do not.
 
\begin{theorem}\label{thm:sqweighted} 
  Let $\Psi\in \sw$ be such that $\supp(\widehat{\Psi}) \subset \{\xi\in\rn: c_1<\abs{\xi}<c_2\}$
  for some $0<c_1<c_2<\infty$.   Given a sequence
  $\bar{z}=\{z_{k,m}\}_{k\in\ent,m\in\ent^n}\subset \rn,$  define
  $\Psi^{\bar{z}}_{k,m}(x)=2^{kn}\Psi(2^{k}(x+z_{k,m}))$ for
  $x\in\rn,$ $m\in\ent^n$ and $k\in\ent.$   Then for every
  $1<p<\infty$,   $w\in A_p,$ and $f\in L^p(w)$, 
\[
\bigg\|\bigg(\sum_{k\in\ent }\abs{\Psi^{\bar{z}}_{k,m}\ast
  f}^2\bigg)^{\frac{1}{2}}\bigg\|_{L^{p}(w)}\lesssim \norm{f}{L^p(w)},  
\]
where the implicit constants depend on $\Psi$ and $[w]_{A_p}$ but are
independent of $m$ and $\bar{z}.$
\end{theorem}

\smallskip

\begin{remark}
  As we noted in the Introduction, Theorem~\ref{thm:sqweighted}
  improves the corresponding unweighted estimate from \cite[Corollary
  1]{MR3200091}.  There, the authors showed that for the particular
  sequence $z_{k,m}=2^{-k}m$, $k\in\ent$ and $m\in\ent^n,$ 
we have that
\[
\bigg\|\bigg(\sum_{k\in\ent }\abs{\Psi^{\bar{z}}_{k,m}\ast
  f}^2\bigg)^{\frac{1}{2}}\bigg\|_{L^{p}}\lesssim
\log(1+\abs{m})\norm{f}{L^p}.
\]
They proved this inequality by means of an unweighted, weak $(1,1)$
inequality and then interpolating with the unweighted $L^2$ estimate.
They prove the weak $(1,1)$ estimate by showing that the operator
satisfies a vector-valued H\"ormander condition; such a condition is
not sufficient for proving weighted norm inequalities:
see~\cite{MR2098100}.     Our proof of Theorem~\ref{thm:sqweighted}
makes use of Rubio de Francia extrapolation, and so does not yield a
weighted endpoint estimate.  We do not know if such an estimate holds,
either with a constant independent of $m$ or with a constant of order
$\log(1+|m|)$. 
\end{remark}

\begin{remark} We refer the reader to Section~\ref{sec:LM} for versions of Theorem~\ref{thm:sqweighted} in the settings of weighted Lorentz spaces and Morrey spaces.
\end{remark}

\medskip

The proof of Theorem~\ref{thm:sqweighted} requires two propositions
which we state here; their proofs are given in the final part of this
section.

\begin{proposition}\label{lem:squareestimate}
  Let $\varphi\in\sw$ be such that $\supp(\widehat{\varphi})\subset\{\xi\in\rn: c_1<\abs{\xi}<c_2\}$ for some $0<c_1<c_2<\infty$ and
  $\sum_{k\in\ent}\abs{\widehat{\varphi}(2^{-k}\xi)}^2=c_\varphi$ for
  some $c_\varphi>0$ and for all $\xi\neq 0$.  Set
  $\varphi_k(x)=2^{kn}\varphi(2^kx)$ for $k\in\ent.$ If $1<p<\infty$
  and $w\in A_p,$ then
 \begin{equation*}
\bigg\|\sum_{k\in\ent} \varphi_k\ast f_k\bigg\|_{L^p(w)}\lesssim  \bigg\|\bigg(\sum_{k\in\ent}\abs{f_k}^2\bigg)^{\frac{1}{2}}\bigg\|_{L^p(w)}
\end{equation*}
for all sequences $\{f_k\}_{k\in\ent}\subset L^p(w)$ such that the
righthand side is finite.  The implicit constant depends only on
$\varphi$ and $[w]_{A_p}$. 
\end{proposition}

\medskip

\begin{proposition}\label{lem:unifestimate} 
Given $\Psi\in\sw$ and a  sequence $\bar{z}=\{z_{k,m}\}_{k\in\ent,m\in\ent^n}\subset \rn,$  define $\Psi^{\bar{z}}_{k,m}(x)=2^{kn}\Psi(2^{k}(x+z_{k,m}))$ for $x\in\rn,$ $m\in\ent^n$ and $k\in\ent.$
If $1<p<\infty$ and $w\in A_p$, then for all $f\in L^p(w)$,
\begin{equation*}
\sup_{k,m\in\ent}\norm{\Psi^{\bar{z}}_{k,m}\ast f}{L^p(w)}\lesssim [w]^{\frac{1}{p}}_{A_p} \norm{f}{L^p(w)}, 
\end{equation*}
where the implicit constant depends on $\Psi$ and is independent of  $\bar{z}$ and $w$.
\end{proposition}

\medskip

\begin{proof}[Proof of Theorem~\ref{thm:sqweighted}] 
Our proof is inspired by the argument
in~\cite{MR837527}, but it is simpler because
of the structure of our operator.  By Rubio de Francia extrapolation
(see~\cite{MR2797562}), it is enough to prove the desired inequality
when $p=2.$ 

Set
$T^{\bar{z}}_{m,\varepsilon}(f)=\sum_{k\in\ent}
\varepsilon_k\Psi^{\bar{z}}_{k,m}\ast f$,
where $m\in\ent^n$ and $\varepsilon=\{\varepsilon_k\}_{k\in\ent}$ with
$\varepsilon_k=\pm 1$ for each $k\in\ent.$ Without loss of generality
we may assume $c_1=\frac{1}{2}$ and $c_2=2$.   Fix $\varphi\in\sw$ such that
$\widehat{\varphi}\equiv 1$ on
$\{\xi\in\rn:\frac{1}{2}< \abs{\xi}< 2\},$ $\widehat{\varphi}$ is
supported on an annulus,  and
$\sum_{k\in\ent}\abs{\widehat{\varphi}(2^{-k}\xi)}^2=c_\varphi$ for all
$\xi\neq 0$ for some constant $c_\varphi.$  Define $\varphi_k(x)=2^{kn}\varphi(2^kx).$ Then
$\varepsilon_k\Psi^{\bar{z}}_{k,m} \ast f= \varphi_k\ast
\varepsilon_k\Psi^{\bar{z}}_{k,m}\ast \varphi_k\ast f$ pointwise
for all $k\in\ent,$  $m\in\ent^n$, and $f\in \mathcal{S}'(\rn).$  By
Propositions~\ref{lem:squareestimate} and~\ref{lem:unifestimate}, and by
Lemma~\ref{lemma:square-estimate}, it follows that for $w\in A_2$,
$f\in L^2(w)$, and all $\bar{z}$ and $\varepsilon$,
\begin{multline*}
\norm{T^{\bar{z}}_{m,\varepsilon}f}{L^2(w)}^2
=\bigg\|\sum_{k\in\ent}\varphi_k\ast
  \varepsilon_k\Psi^{\bar{z}}_{k,m}\ast \varphi_k\ast f\bigg\|_{L^2(w)}^2 \\
\lesssim \sum_{k\in\ent}\norm{\Psi^{\bar{z}}_{k,m}\ast \varphi_k\ast f}{L^2(w)}^2
\lesssim \sum_{k\in\ent}\norm{\varphi_k\ast f}{L^2(w)}^2\lesssim \norm{f}{L^2(w)}^2;
\end{multline*}
the implicit constants depend only on $\varphi$, $\Psi$, and
$[w]_{A_2}$. 
In other words, we have shown that $T^{\bar{z}}_{m,\varepsilon}$ is
bounded on  $L^2(w)$ for every $w\in A_2$ with the
operator norm controlled by a constant independent of $m,$
$\varepsilon$ and $\bar{z}.$ Clearly, the same argument shows that 
the operators $T^{\bar{z},+}_{m,\varepsilon}$ and
$T^{\bar{z},-}_{m,\varepsilon}$ defined as
$T^{\bar{z}}_{m,\varepsilon}$ but with $k\in\na_0$ and $-k\in\na,$
respectively, instead of $k\in\ent$, are bounded on $L^2(w)$ with
constants independent of the same quantities.

We can now argue as in \cite[p. 177]{MR1800316}.  Let
$\{r_k\}_{k\in\na_0}$ be the system of Rademacher functions.
That is, define  $r_0(t)=-1$ for $0\le t<\frac{1}{2},$ $r_0(t)=1$ for
$\frac{1}{2}\le t<1$, and extend it as a periodic function on $\re$.
Then define $r_k(t)=r_0(2^kt)$ for $0\le t<1.$ Recall
that $\{r_k\}_{k\in \na_0}$ is an orthonormal system in $L^2([0,1])$
and  we have that for all $\{a_k\}\in\ell^2$,
\[
\bigg\|\sum_{k=0}^\infty a_k r_k\bigg\|_{L^2([0,1])}
= \left(\sum_{k=0}^\infty \abs{a_k}^2\right)^{\frac{1}{2}}.
\]
Therefore, if we set $\varepsilon_t=\{r_k(t)\}_{k\in\na_0}$ for $0\le
t<1,$ then we have that
\begin{align*}
\sum_{k\in\ent} \abs{\Psi^{\bar{z}}_{k,m}\ast f(x)}^2 &=\int_0^1\abs{\sum_{k=0}^\infty r_k(t) \Psi^{\bar{z}}_{k,m}\ast f(x)}^2\,dt+\int_0^1\abs{\sum_{k=-\infty}^1 r_{-k}(t) \Psi^{\bar{z}}_{k,m}\ast f(x)}^2\,dt\\
&= \int_0^1\abs{T^{\bar{z},+}_{m,\varepsilon_t}f(x)}^2\,dt +\int_0^1\abs{T^{\bar{z},-}_{m,\varepsilon_t}f(x)}^2\,dt.                                                                   
\end{align*}
If we fix $w\in A_2$ and compute the $L^2(w)$ norm, then we get 
\begin{multline*}
\bigg\|\bigg(\sum_{k\in\ent} \abs{\Psi^{\bar{z}}_{k,m}\ast
  f(x)}^2\bigg)^{\frac{1}{2}}\bigg\|_{L^2(w)}^2 \\
=  \int_0^1\norm{T^{\bar{z},+}_{m,\varepsilon_t}f}{L^2(w)}^2\,dt+\int_0^1\norm{T^{\bar{z},-}_{m,\varepsilon_t}f}{L^2(w)}^2\,dt\lesssim \norm{f}{L^2(w)}^2,
\end{multline*}
with the implicit constant independent of $m,$ $\bar{z}$ and $f.$
\end{proof}

\subsection*{Proof of Propositions~\ref{lem:squareestimate} and \ref{lem:unifestimate}}

\begin{proof}[Proof of Proposition~\ref{lem:squareestimate}]
Fix $w,$  $\varphi$ and $\{f_k\}_{k\in\ent}$  as in the hypotheses.
As before, we can assume without loss of generality that in defining
the support of $\widehat{\varphi},$  $c_1=\frac{1}{2}$ and $c_2=2$.
For all $N\in \na_0$, define
\[ F_N=\sum_{\abs{\ell}\le N} \varphi_\ell\ast f_\ell.  \]
Then for all $N_2>N_1$ and all $k\in \ent$ we have that 
\[ \varphi_k\ast (F_{N_2}-F_{N_1})=\sum_{N_1<\abs{\ell}\le
  N_2}\varphi_k\ast \varphi_\ell\ast f_\ell. \] 
Comparing supports on the Fourier transform side, it follows that
\[
\varphi_k\ast (F_{N_2}-F_{N_1})=
\begin{cases}
  0   & \text{ if } \abs{k}\le N_1-1 \text{  or } \abs{k}\ge N_2+2,   \\
  \sum_{\ell=-1}^1\varphi_k\ast \varphi_{k-\ell}\ast f_{k-\ell}   &   \text{ if }  N_1+2\le \abs{k}\le N_2-1,  \\ 
   \sum_{\ell=-1}^0\varphi_k\ast \varphi_{k-\ell}\ast f_{k-\ell}   &   \text{ if }   \abs{k}= N_1+1,  \\ 
    \sum_{\ell=0}^1 \varphi_k\ast \varphi_{k-\ell}\ast f_{k-\ell}   &   \text{ if }   \abs{k}= N_2 , \\ 
   \varphi_k\ast \varphi_{k+1}\ast f_{k+1}   &   \text{ if }  \abs{k}= N_1,  \\
    \varphi_k\ast \varphi_{k-1}\ast f_{k-1}   &   \text{ if }   \abs{k}= N_2+1.
  \end{cases}
\]

Recall that for all $k\in \ent$,
$\abs{\varphi_k\ast g}\lesssim \mathcal{M}(g)$ with a constant that
depends only on $\varphi$.  (See~\cite{MR1800316}.)  Since
$\mathcal{M}$ is bounded on $L^p(w)$, it follows that $F_N\in L^p(w)$ for
all $N$.  Moreover, by  Lemma~\ref{lemma:square-estimate}, this
maximal operator inequality and Lemma~\ref{lemma:fefferman-stein}, we
get that
\begin{multline*}
\norm{F_{N_1}-F_{N_2}}{L^p(w)}
\lesssim \sum_{\ell=-1}^1\bigg\|\bigg(\sum_{N_1\le \abs{k}\le N_2+1}
\abs{\varphi_k\ast \varphi_{k-\ell}\ast
  f_{k-\ell}}^2\bigg)^{\frac{1}{2}}\bigg\|_{L^p(w)}\\ 
\lesssim \sum_{\ell=-1}^1
\bigg\|\bigg(\sum_{N_1\le \abs{k}\le
  N_2+1}\abs{\mathcal{M}(\mathcal{M}
  (f_{k-\ell}))}^2\bigg)^{\frac{1}{2}}\bigg\|_{L^p(w)}
\lesssim \bigg\|\bigg(\sum_{N_1-1\le \abs{k}\le
  N_2+2}\abs{f_k}^{2}\bigg)^{\frac{1}{2}}\bigg\|_{L^p(w)}. 
\end{multline*}
The last term in the above chain of inequalities converges to 0 as
$N_1,N_2$ tend to infinity; therefore,  we have that $\{F_N\}_{N\in\na_0}$ is
a Cauchy sequence in $L^p(w)$.  Hence,
$\sum_{k\in\ent} \varphi_k\ast f_k$ converges in $L^p(w).$ Moreover,
the same argument as before also shows that
\[
\bigg\|\sum_{\abs{k}\le N}\varphi_k\ast f_k\bigg\|_{L^p(w)}
\lesssim \bigg\|\bigg(\sum_{ \abs{k}\le
  N+2}\abs{f_k}^{2}\bigg)^{\frac{1}{2}}\bigg\|_{L^p(w)}.
\]
with the implicit constant independent of the sequence
$\{f_k\}_{k\in\ent}.$ If we let $N$ tend to infinity we get the
desired inequality.
\end{proof}

\medskip

If we did not care about the size of the constant, we could easily
prove a version of Proposition~\ref{lem:unifestimate} using the pointwise
inequality
$\abs{\Psi^{\bar{z}}_{k,m}\ast f(x)}\le C \mathcal{M}(f)(x)$.
However, the constant may
depend on $\bar{z}$ and blow up with $k$ and $m.$ Moreover, the
uniform pointwise estimate
$\abs{\Psi^{\bar{z}}_{k,m}\ast f(x)}\le C_\Psi
\mathcal{M}(f)(x-z_{k,m})$
only gives the desired result if $w$ is constant since otherwise
$L^p(w)$ is not translation invariant.  

Therefore, to prove the desired uniform
estimate in weighted spaces  we will use an argument developed
in~\cite{CUMR} for matrix weights. The proof requires two lemmas.
The first was proved in~\cite{CUMR};
for the convenience of the reader we give the proof for the scalar case here.

\begin{lemma}\label{lem:averageballs}
If $1<p<\infty$ and $w\in A_p$,  then for every ball $B$ and $f\in L^p(w)$,
\[
\norm{\abs{B}^{-1}(\chi_B\ast f)}{L^p(w)}\lesssim [w]^{\frac{1}{p}}_{A_p}\norm{f}{L^p(w)},
\]
with the implicit constant independent of  $w.$ 
\end{lemma}

\begin{proof} 
 Let $p,$ $w$ and $f$ be as in the hypotheses.  It is enough to prove this estimate for cubes with sides parallel to
  the coordinate axes instead of balls: given any ball $B$, if $Q$ is
  the smallest cube that contains $B$, then 
  $\abs{\abs{B}^{-1}(\chi_B\ast f)}\le C_n \abs{Q}^{-1}(\chi_Q\ast
  \abs{f})$.   

Fix a cube $Q$, denote its side length by $\ell(Q)$ and define the cubes  
$ Q_m=Q+\ell(Q) m$ for $m\in\ent^n.$ Then  $\{Q_m\}_{m\in\ent^n}$ is a
pairwise disjoint partition of $\rn.$ Further, we can divide the
family  $\{3Q_m\}_{m\in\ent^n}$ into $3^n$ families $\Q_{J},$  $J\in
\{1,\cdots, 3^n\},$ of pairwise disjoint cubes ($3Q_m$ denotes the
cube with the same center as $Q_m$ and side length $3\ell(Q_m)).$ Then
for all $x\in Q_m$, 
\[
\abs{Q}^{-1}\abs{(\chi_Q\ast f)(x)}\le \abs{Q}^{-1}\int_\rn\abs{f(y)\chi_Q(x-y)}\,dy\le 3^n\avgint_{3Q_m} \abs{f(y)}\,dy.
\]
Hence,
\begin{align*}
&\int_{\rn}\abs{\abs{Q}^{-1}(\chi_Q\ast f)(x)}^p w(x)\,dx \\
&\qquad \qquad =\sum_{m\in\ent^n} \int_{Q_m}\abs{\abs{Q}^{-1}(\chi_Q\ast f)(x)}^p w(x)\,dx\\
&\qquad \qquad\lesssim \sum_{m\in\ent^n} \int_{Q_m}\left(\avgint_{3Q_m} \abs{f(y)}\,dy\right)^p w(x)\,dx\\
&\qquad \qquad\lesssim \sum_{J=1}^{3^n}\sum_{3Q_m\in \Q_J}\int_{3Q_m}w(x)\,dx \left(\avgint_{3Q_m}\abs{f(y)}w(y)^{\frac{1}{p}}w(y)^{-\frac{1}{p}}dy\right)^p.
\intertext{ If we apply H\"older's inequality, and  use the facts that
  $w\in A_p$ and the cubes in $\Q_J$ are pairwise disjoint, we get}
&\qquad \qquad\lesssim \sum_{J=1}^{3^n}\sum_{3Q_m\in \Q_J} \left(\int_{3Q_m}
  \abs{f(x)}^pw(x)\,dx \right)\left(\avgint_{3Q_m}w(x)\,dx\right)
  \left(\avgint_{3Q_m}w(x)^{-\frac{p'}{p}}\,dx \right)^{\frac{p}{p'}} \\
&\qquad \qquad \lesssim [w]_{A_p}\norm{f}{L^p(w)}^p.
\end{align*}
The implicit constants are independent of $w$ and $f$, and the proof
is complete.
\end{proof}

The next lemma is also from~\cite{CUMR} where it was given implicitly
and without proof.

\begin{lemma}\label{lem:radialdecr}
Let $\Phi\in L^1(\rn)$ be non-negative and radially decreasing. Define
$\tau_z(\Phi)(x)=\Phi(x+z)$ for $z\in \rn.$  If $1<p<\infty$ and $w\in
A_p$,  then for all $f\in L^p(w)$, 
\[
\norm{\tau_z(\Phi)\ast f}{L^p(w)}\lesssim \norm{\Phi}{L^1}  [w]^{\frac{1}{p}}_{A_p}\norm{f}{L^p(w)},
\]
with the implicit constant independent of  $z,$  $w$ and  $\Phi.$ 
\end{lemma}

\begin{proof}
Assume first that $\Phi\in L^1(\rn)$ is of the form
\[
\Phi(x)=\sum_{k=1}^\infty a_{k}\abs{B_k}^{-1}\chi_{B_k}(x),
\]
where $a_k\ge 0$ and $B_k$ is a ball centered at the origin for each
$k\in \na.$ Note that $\norm{\Phi}{L^1}=\sum_{k=1}^\infty a_k.$  Fix
$z\in\rn$; then we have
\[\tau_{z}(\Phi)=\sum_{k=1}^\infty a_k \, |\widetilde{B_k}|^{-1}\chi_{\widetilde{B_k}},\]
where $\widetilde{B_k}=-z+B_k$ for $k\in \na.$ With $p,$ $w$ and $f$
as in the hypotheses, by Lemma~\ref{lem:averageballs} we have that
\begin{multline*}
\norm{\tau_z(\Phi)\ast f}{L^p(w)}
\le \sum_{k=1}^\infty a_k \, \norm{|\widetilde{B_k}|^{-1}(\chi_{\widetilde{B_k}}\ast f)}{L^p(w)}\\
\lesssim [w]_{A_p}^{\frac{1}{p}}\sum_{k=1}^\infty a_k \norm{f}{L^p(w)}=[w]_{A_p}^{\frac{1}{p}}\norm{\Phi}{L^1} \norm{f}{L^p(w)},
\end{multline*}
where the implicit constant is independent of $z,$ $w,$ $\Phi$ and $f.$

To complete the proof, note that an arbitrary function as in the hypotheses can
be approximated from below by a sequence of functions of the form
treated above.  The desired inequality then follows by the Monotone Convergence Theorem.
\end{proof}

\begin{proof}[Proof of Proposition~\ref{lem:unifestimate}]
 
  Since $\Psi\in \sw,$ there exists a non-negative radially decreasing
  function $\Phi\in L^1(\rn)$ such that $\abs{\Psi(x)}\le \Phi(x)$ for
  all $x\in \rn.$ Let
  $\Phi^{\bar{z}}_{k,m}(x)=2^{nk} \Phi(2^k(x+z_{k,m}))$.
Fix  $1<p<\infty$ and $w\in A_p$;  then by Lemma~\ref{lem:radialdecr}
we have that 
\begin{multline*}
\norm{\Psi^{\bar{z}}_{k,m}\ast f}{L^p(w)}
\le \norm{\Phi^{\bar{z}}_{k,m}\ast \abs{f}}{L^p(w)}\\
\lesssim \norm{2^{kn}\Phi(2^{k}\cdot)}{L^1}  [w]^{\frac{1}{p}}_{A_p}\norm{f}{L^p(w)}=\norm{\Phi}{L^1}  [w]^{\frac{1}{p}}_{A_p}\norm{f}{L^p(w)}.
\end{multline*}
\end{proof}

\section{Bilinear extrapolation on variable Lebesgue spaces}\label{sec:variable}

The main result in this section is Theorem~\ref{thm:extrapolation}, a
bilinear extrapolation result  that allows to deduce bilinear
estimates in variable Lebesgue spaces from bilinear estimates in
weighted Lebesgue spaces.  This result is key for our proof of
Theorem~\ref{thm:KPvariable}.   Extrapolation is an important
technique for proving norm inequalities in variable Lebesgue spaces:
we refer the reader to \cite{MR3026953,MR2210118} for further details
in the linear case.  

We divide this section into two
parts.  In the first we give some basic definitions and results about
variable Lebesgue spaces.  In the second we state and prove
Theorem~\ref{thm:extrapolation}.  In Section~\ref{sec:applextrap} below, in order to
illustrate the broader utility of bilinear extrapolation, we give 
applications of  Theorem~\ref{thm:extrapolation} to prove estimates for a variety of bilinear
operators. 
 
\subsection*{Definitions and preliminary results}
For complete information on variable Lebes\-gue spaces and for proofs of
the results stated here, see~\cite{MR3026953,MR2790542}

Define the collection $\Pp$ of exponent functions to be the set of
measurable functions $\pp : \rn \rightarrow [1,\infty]$.  Similarly,
let  $\Pp_0$ denote the set of measurable functions $\pp : \rn
\rightarrow (0,\infty]$.    Given an exponent $\pp$, we let
\[ p_- = \essinf_{x\in \rn} p(x), \qquad p_+ = \esssup_{x\in \rn}
p(x). 
\]
Given  $p(\cdot)\in\Pp_0,$ we define the modular
\[ \rho_\pp(f) 
= \int_{\rn\setminus \re^n_\infty} \abs{f(x)}^{p(x)}\,dx + \|f\|_{L^\infty(\re^n_\infty)}, \]
where $\re^n_\infty = \{ x\in \rn : p(x)=\infty\}$.  
The variable Lebesgue space $L^{p(\cdot)}$
consists of all measurable functions $f$ defined on $\rn$ that satisfy
$\rho_\pp(f/\lambda)<\infty$ 
for some $\lambda>0.$ 
This is a quasi-Banach space (Banach
space if $p(\cdot)\in \Pp$) with the quasi-norm (norm if
$p(\cdot)\in \Pp$) given by
\[
\norm{f}{L^{p(\cdot)}}=\inf\left\{\lambda>0: \rho_\pp(f/\lambda)\le 1\right\}.
\]
The variable Lebesgue spaces generalize the classical Lebesgue spaces:
 if $0<p_0\leq \infty$ and $p(\cdot)\equiv p_0,$ then
$L^{p(\cdot)}=L^{p_0}$ with equality of norms.

The following lemmas are basic properties of the norm.  The first
relates the norm and the modular, the second generalizes H\"older's
inequality, and the third gives an equivalent expression for the norm.

\begin{lem} \label{lemma:modular-norm}
Given $\pp \in \Pp_0$, $\|f\|_{L^{\pp}} \leq 1$ if and only if
$\rho_\pp(f)\leq 1$. 
\end{lem}

\begin{lem} \label{lemma:gen-holder}
Given $\pp,\, \qq,\, \rr \in \Pp_0$, suppose
\[ \frac{1}{\rr} = \frac{1}{\pp} + \frac{1}{\qq}.  \]
Then  for all $f\in L^\pp$ and $g\in L^\qq$, 
\[ \|fg\|_{L^\rr} \lesssim \|f\|_{L^\pp}\|g\|_{L^\qq}. \]
The implicit constant depends only on $\pp$ and $\qq$.  
\end{lem}

\begin{remark} 
In~\cite{MR3026953} both are proved assuming that the
exponents are in $\Pp$; however, the proofs can be easily adapted to
this more general setting.
\end{remark}

\begin{lem} \label{lemma:dual}
If $\pp \in \Pp$, then for all $f\in \Lp$,
\[ \|f\|_{\Lp} \approx \sup \int_{\subRn} fg\,dx, \]
where the supremum is taken over all $g\in L^\cpp$,
$\|g\|_{L^\cpp}=1$, with $\cpp \in \Pp$ defined pointwise by
\[ \frac{1}{\pp} + \frac{1}{\cpp} = 1. \]
The implicit constants depend only on $\pp$. 
\end{lem}

Central to our result is the boundedness of the maximal operator
$\mathcal{M}$ on $\Lp$.  Let $\B$ be the family of all $\pp \in
\Pp$ such that  for all $f\in \Lp$, 
\[ 
\|\mathcal{M}f\|_{L^\pp} \lesssim \|f\|_{L^\pp}.
\]
The norm of $\mathcal{M}$ as a bounded operator on $L^{p(\cdot)}$ will
be denoted by $\norm{\mathcal{M}}{p(\cdot)}.$ A necessary condition
for $\pp \in \B$ is that $p_->1$.    As a consequence, our hypothesis
in Theorems~\ref{thm:KPvariable} and~\ref{thm:extrapolation} that
$\left(\frac{p(\cdot)}{p}\right)',\,\left(\frac{q(\cdot)}{q}\right)'\in\B$
immediately implies that $p_+,\,q_+<\infty$.

A sufficient condition for
$\pp \in \B$ is that $\pp$ is log-H\"older continuous locally: there
exists $C_0>0$ such that
\begin{equation}\label{eq:locallogH}
|p(x) - p(y)| \leq \frac{C_0}{-\log(|x-y|)}, \qquad
|x-y|<\frac{1}{2}; 
\end{equation}
and log-H\"older continuous at infinity:  there exists $p_\infty,\,C_\infty>0$
such that
\begin{equation}\label{eq:inflogH}
 |p(x)-p_\infty| \leq \frac{C_\infty}{\log(e+|x|)}. 
 \end{equation}

 While not strictly necessary to prove
 Theorem~\ref{thm:extrapolation}, we want to note the following
 result, which in practice makes it easier to apply.  Given
 $p(\cdot)\in \mathcal{P}$ with $1<p_{-}\leq p_{+}<\infty$, the
 following statements are equivalent:
\begin{enumerate}
\item $p(\cdot)\in \mathcal{\B},$
\item $p'(\cdot)\in \mathcal{\B},$
\item $p(\cdot)/q\in \mathcal{\B}$ for some $1<q<p_{-},$
\item $(p(\cdot)/q)'\in \mathcal{\B}$ for some $1<q<p_{-}.$
\end{enumerate}

\subsection*{Statement and proof of bilinear extrapolation}

\begin{theorem}\label{thm:extrapolation} 
Let  $\Ff$ be a family of ordered triples of non-negative, measurable
functions defined on $\rn.$ Suppose there are indices $0<p,\, q,\,r<\infty$
satisfying $\frac{1}{r}=\frac{1}{p}+\frac{1}{q}$,  such that for
every $v,w\in A_1$,
\begin{multline}\label{eq:extrapolation1}
\left(\int_\rn h(x)^r v(x)^{\frac{r}{p}}w(x)^{\frac{r}{q}} dx \right)^{\frac{1}{r}} \\
\lesssim  \left(\int_\rn f(x)^{p} v(x) dx \right)^{\frac{1}{p}}
\left(\int_\rn g(x)^{q} w(x) dx \right)^{\frac{1}{q}} 
\end{multline}
for all $(h,f,g)\in \Ff$ such that the lefthand side is finite and
where the implicit constant depends only on $p,$ $q$, $[w]_{A_1}$ and
$[v]_{A_1}$.  Let
$p(\cdot),q(\cdot),r(\cdot)\in \Pp_0(\rn)$ be such that
$\frac{1}{\rr}=\frac{1}{\pp}+\frac{1}{\qq}$, $0<p<p_{-},$ $0<q<q_{-},$ and
$\left(\frac{p(\cdot)}{p}\right)',\,\left(\frac{q(\cdot)}{q}\right)'\in\B.$
Then
\begin{equation}\label{eq:extrapolation2}
\norm{h}{L^{r(\cdot)}} \lesssim \norm{f}{L^{p(\cdot)}}
\norm{g}{L^{q(\cdot)}} 
\end{equation}
for all $(h,f,g)\in\Ff$  such that  $h\in L^{r(\cdot)}$.  The implicit
constant only depends on $\pp$ and $\qq$.
\end{theorem}

\begin{proof} 
  Set $\bar{r}(\cdot)=\frac{r(\cdot)}{r},$
  $\bar{p}(\cdot)=\frac{p(\cdot)}{p}$ and
  $\bar{q}(\cdot)=\frac{q(\cdot)}{q}$.   Then by our hypotheses,
  $1<\bar{p}_- \leq \bar{p}_+<\infty$, and the same is true for
  $\bar{r}(\cdot)$ and $\bar{q}(\cdot)$.

We first define two Rubio de
  Francia iteration algorithms:   given a non-negative function $\tau$,
\begin{equation*}
\Rf_1 \tau(x) = \sum_{k=0}^\infty \frac{\mathcal{M}^k \tau(x)}{2^k
  \|\mathcal{M}\|_{\cbp}^k}, \quad\quad
\Rf_2 \tau(x) = \sum_{k=0}^\infty \frac{\mathcal{M}^k \tau(x)}{2^k
  \|\mathcal{M}\|_{\cbq}^k}. 
\end{equation*}
Then 
$\tau(x) \leq \Rf_1 \tau(x)$ and, since  $\mathcal{M}$ is bounded on $L^\cbp$, we have that  $\|\Rf_1 \tau \|_\cbp\leq
2\|\tau\|_\cbp$ and $\Rf_1 \tau \in A_1$ with $[\Rf_1 \tau]_{A_1}
\leq 2 \|\mathcal{M}\|_{\cbp}$.  (See~\cite[Theorem~5.24]{MR3026953}
for details.)  The same is true for $\mathcal{R}_2\tau$ with $\bp$ replaced by
$\bq$ everywhere.

Now fix a triple $(h,f,g)\in \Ff$
  such that $\|h\|_{L^\rr} < \infty$.  Then by Lemma~\ref{lemma:dual},
\[
\norm{h}{L^{r(\cdot)}}^r=\norm{h^r}{L^{\bar{r}(\cdot)}}\approx \sup\int_{\rn}h(x)^r\tau(x)\,dx,
\]
where the supremum is taken over all non-negative functions $\tau\in
L^{\bar{r}'(\cdot)}$ with $\norm{\tau}{L^{\bar{r}'(\cdot)}}=1.$
Therefore, it is enough to prove that for all such $\tau,$
\[
\int_{\rn}h(x)^r\tau(x)\,dx\lesssim \norm{f}{L^{p(\cdot)}}^r \norm{g}{L^{q(\cdot)}}^r,
\]
where the implicit constant is independent of $\tau.$ 

Define functions $\theta_1(\cdot),\,\theta_2(\cdot)$ by
\[ \theta_1(\cdot) = \frac{r\cbr}{p\cbp}, \qquad 
\theta_2(\cdot) = \frac{r\cbr}{q\cbq}. \]
Then for all $x$,
\[ \theta_1(x)+\theta_2(x) 
= \bar{r}'(x)\left( \frac{r}{p}
  \frac{p(x)-p}{p(x)}+\frac{r}{q}\frac{q(x)-q}{q(x)}\right) 
= \bar{r}'(x) \left( \frac{r}{p}+\frac{r}{q}-\frac{r}{r(x)}\right) =
1. \]
Hence, by the properties of the iteration algorithms,
\begin{multline*}
\int_{\rn}h(x)^r\tau(x)\,dx \\
= \int_\rn h(x)^r \tau(x)^{\theta_1(x)} \tau(x)^{\theta_2(x)}\,dx 
\le\int_{\rn}h(x)^r \Rf_1(\tau^{\frac{\cbr}{\cbp}})
(x)^{\frac{r}{p}}
\Rf_2(\tau^{\frac{\cbr}{\cbq}})(x)^{\frac{r}{q}}\,dx.
\end{multline*}
We claim that the righthand side of this inequality is finite.  To see
this, first note that by the computation above for
$\theta_1+\theta_2$, we have that
\[1= \frac{1}{\br}+\frac{1}{\cbr}
=
\frac{1}{\br}+\frac{r}{p}\frac{1}{\cbp}+\frac{r}{q}\frac{1}{\cbq}. \]
Therefore, by Lemma~\ref{lemma:gen-holder},
\begin{align*}
& \int_{\rn}h(x)^r \Rf_1(\tau^{\frac{\cbr}{\cbp}})
(x)^{\frac{r}{p}}
\Rf_2(\tau^{\frac{\cbr}{\cbq}})(x)^{\frac{r}{q}}\,dx \\
& \qquad \lesssim \|h^r\|_{L^\br} 
\|\Rf_1(\tau^{\frac{\cbr}{\cbp}})^{\frac{r}{p}}\|_{L^{\frac{p}{r}\cbp}}
\|\Rf_2(\tau^{\frac{\cbr}{\cbq}})^{\frac{r}{q}}\|_{L^{\frac{q}{r}\cbq}} \\
&\qquad \lesssim \|h\|^r_{L^\rr} 
\|\Rf_1(\tau^{\frac{\cbr}{\cbp}})\|^{\frac{r}{p}}_{L^\cbp}
\|\Rf_2(\tau^{\frac{\cbr}{\cbq}})\|^{\frac{r}{q}}_{L^\cbq} \\
&\qquad \lesssim \|h\|_{L^\rr}^r 
\|\tau^{\frac{\cbr}{\cbp}}\|^{\frac{r}{p}}_{L^\cbp}
\|\tau^{\frac{\cbr}{\cbq}}\|^{\frac{r}{q}}_{L^\cbq}. 
\end{align*}
Since $\|\tau\|_{L^\cbr} =1$, by Lemma~\ref{lemma:modular-norm} we have that
\[ 1\geq \int_\rn \tau(x)^{\bar{r}'(x)}\,dx 
= \int_\rn
\left(\tau(x)^{\frac{\bar{r}'(x)}{\bar{p}'(x)}}\right)^{\bar{p}'(x)}\,dx, \]
which again by Lemma~\ref{lemma:modular-norm}, since $(\bar{p}')_+<\infty$, implies that $\|\tau^{\frac{\cbr}{\cbp}}\|_{L^\cbp}\leq 1$.
Similarly, we have that $\|\tau^{\frac{\cbr}{\cbq}}\|_{L^\cbq}\leq 1$.
Therefore, since $\|h\|_{L^\rr} <\infty$, it follows that the righthand
side is finite.  

Given this, and given that $\Rf_1(\tau^{\frac{\cbr}{\cbp}}),\,
\Rf_2(\tau^{\frac{\cbr}{\cbq}})\in A_1$ with $A_1$
characteristics independent of $\tau$,  we can apply our hypothesis to
get 
\begin{multline*}
\int_{\rn}h(x)^r \Rf_1(\tau^{\frac{\cbr}{\cbp}})
(x)^{\frac{r}{p}}
\Rf_2(\tau^{\frac{\cbr}{\cbq}})(x)^{\frac{r}{q}}\,dx \\
\lesssim \left(\int_{\rn}f(x)^p
  \Rf_1(\tau^{\frac{\cbr}{\cbp}})(x)\,dx\right)^{\frac{r}{p}}
 \left(\int_{\rn}g(x)^q \Rf_2(\tau^{\frac{\cbr}{\cbq}})
   (x)\,dx\right)^{\frac{r}{q}}. 
\end{multline*}
To estimate the first integral on the righthand side we apply
Lemma~\ref{lemma:gen-holder}:
\begin{multline*}
\int_{\rn}f(x)^p  \Rf_1(\tau^{\frac{\cbr}{\cbp}})(x)\,dx \\
 \lesssim \|f^p\|_{L^\bp} \|\Rf_1(\tau^{\frac{\cbr}{\cbp}})\|_{L^\cbp} 
 \lesssim \|f\|_{L^\pp}^p \|\tau^{\frac{\cbr}{\cbp}}\|_{L^\cbp} 
\lesssim \|f\|_{L^\pp}^p. 
\end{multline*}
In exactly the same way we have that the second integral is bounded by 
$\|g\|_\qq^q$.  If we combine all of the above estimates, we get the
desired inequality and this completes the proof.
\end{proof}

\section{Further applications of Theorem~\ref{thm:extrapolation}}\label{sec:applextrap}

In this section we show that Theorem~\ref{thm:extrapolation}
implies boundedness properties in variable Lebesgue spaces for a
variety of bilinear operators.

In order to apply Theorem~\ref{thm:extrapolation} to a bilinear
operator $T$, the corresponding family $\Ff$ will consist of triples
of the form $(|T(f,g)|,|f|,|g|)$, where $f,\,g$ are chosen from the
domain of $T$ or some appropriate (dense) subset of the domain.  To
insure {\em a priori} that the assumption
$\|T(f,g)\|_{L^{r(\cdot)}}<\infty$ holds, it will suffice to replace
$|T(f,g)|$ in each triple by $\min(|T(f,g)|,N) \chi_{B(0,N)} $ and
then apply the the monotone convergence theorem for variable Lebesgue
spaces (see~\cite[Theorem~2.59]{MR3026953}).

\subsection*{Bilinear Calder\'on-Zygmund operators in variable
  Lebesgue spaces}\label{sec:cz}
Let $K(x,y,z)$ be a complex-valued, locally integrable function on
$\re^{3n}\setminus \triangle,$ where $\triangle=\{(x,x,x):x\in\rn\}.$
$K$ is a Calder\'on-Zygmund kernel if there exist $A>0$ and $\delta>0$
such that for all $(x,y,z)\in \re^{3n}\setminus \triangle$,
\begin{equation*}\label{size}
\abs{K(x, y, z)} \le \frac{A}{(\abs{x-y}+\abs{x-z}+\abs{y-z})^{2n}}
\end{equation*}
and 
\begin{equation*}
\abs{K(x, y, z) - K(\tilde{x}, y, z)} \le  \frac{A\,\abs{x-\tilde{x}}^{\delta}}{(\abs{x-y}+\abs{x-z}+\abs{y-z})^{2n+\delta}}
\end{equation*}
whenever
$\abs{x - \tilde{x}} \leq \frac{1}{2} \max ( \abs{x-z}, \abs{x-y})$.
We also assume that the two analogous difference estimates with respect to the variables $y$ and
$z$ hold. An operator $T,$ continuous from $\sw\times \sw$ into $\swp,$ is
a bilinear Calder\'on-Zygmund operator if it satisfies
two conditions: 
\begin{enumerate}

\item there exists a bilinear Calder{\'o}n-Zygmund kernel
$K$ such that
    \begin{equation*}
    T(f,g)(x) = \int_{\re^{2n}} K(x, y, z) f(y) g(z) \,dy \, dz
    \end{equation*}
for all $f,\,g \in C_c^\infty(\re^n)$ and all $x \notin \text{supp}(f) \cap \text{supp}(g);$

\item  there exist $1\le p,q, r<\infty$  such that $\frac{1}{r}=\frac{1}{p}+\frac{1}{q}$ and $T$ can be extended to a  bounded operator from $L^{p} \times L^{q}$ into $L^r.$ 
\end{enumerate}

Bilinear Calder\'on-Zygmund operators enjoy boundedness properties in
various function spaces. We refer the reader
to~\cite{MR1880324,MR2483720} and the references they contain for more
information on bilinear \cz theory.  To apply
Theorem~\ref{thm:extrapolation} we need a weighted norm inequality
from~\cite [Corollary 8.2]{MR2030573} (see also \cite{MR1947875}).
This is not the best result known, but it enough for our purposes; we
refer the reader to~\cite[Corollary 3.9]{MR2483720} for further
results in the weighted setting.

\begin{thm} \label{thm:czweighted} 
Let $T$ be a bilinear  \cz operator.  Given $1< p,q< \infty,$
$\frac{1}{2}< r<\infty,$  $\frac{1}{r}=\frac{1}{p}+\frac{1}{q},$
suppose $v\in A_{p}$ and $w\in A_q.$ Then for all $ f\in L^p(v),g\in L^q(w),$
\[
\norm{T(f,g)}{L^r(v^{\frac{r}{p}}w^{\frac{r}{q}})}\lesssim \norm{f}{L^p(v)}\norm{g}{L^q(w)},
\]
where the implicit constant depends only on $p,$ $q$, $[v]_{A_p}$,
$[w]_{A_q}$ and the size of the constants for the kernel of $T.$
\end{thm}

As an immediate consequence of Theorem~\ref{thm:czweighted} and
Theorem~\ref{thm:extrapolation} we get the boundedness
of bilinear Calder\'on-Zygmund operators in variable
Lebesgue spaces.

\begin{corollary}\label{coro:czvariable}
Let
$p(\cdot),q(\cdot),r(\cdot)\in \Pp_0(\rn)$ satisfy
$\frac{1}{\rr}=\frac{1}{\pp}+\frac{1}{\qq}$, $p_{-}>1,$ $q_{-}>1,$
$r_{-}>\frac{1}{2}$.  Suppose further that there exist $1<p<p_{-}$ and $1<q<q_{-}$ such that
$\left(\frac{p(\cdot)}{p}\right)',\,\left(\frac{q(\cdot)}{q}\right)'\in\B.$
If $T$ is a bilinear Calder\'on-Zygmund operator,  then for all $ f\in  L^{p(\cdot)},\, g\in L^{q(\cdot)}$,
\begin{equation}\label{eq:extrapolation2}
\norm{T(f,g)}{L^{r(\cdot)}} \lesssim \norm{f}{L^{p(\cdot)}}
\norm{g}{L^{q(\cdot)}}.
\end{equation}
\end{corollary}

\begin{remark}  
The  boundedness of bilinear Calder\'on-Zygmund operators in
  variable Lebesgue spaces was first proved in \cite[Corollary
  2.1]{MR2606534} and \cite[Theorem 3.1]{MR3215477}. We improve upon
  both of these results:  the former
  requires the additional hypothesis $\frac{r(\cdot)}{r}\in\B$ for
  some $0<r<r_{-}$ while the latter assumes $r(\cdot)\in \B.$ In both
  cases, the proofs use linear extrapolation in the scale of
  variable Lebesgue spaces.
\end{remark}

\subsection*{Bilinear multiplier operators in variable Lebesgue space.}
Examples of bilinear Calder\'on-Zygmund operators are  Coifman-Meyer multiplier operators (see \cite{MR518170, MR1880324}). Such operators are of the form 
\[
T_\sigma(f,g)(x)=\int_{\rtn}\sigma(\xi,\eta)\widehat{f}(\xi)\widehat{g}(\eta)\eixxe\dxe,\quad x\in\rn,
\]
where $\sigma$ is a complex-valued, smooth function defined for $\xi,\eta\in\rn,$ called a (bilinear) symbol or multiplier, that satisfies
\begin{equation} \label{eq:CMcondition}
\abs{\partial_\xi^\beta\partial_\eta^\gamma\sigma(\xi,\eta)}\lesssim (\abs{\xi}+\abs{\eta})^{-\abs{\beta+\gamma}} \qquad \xi,\eta\in \rn, (\xi,\eta)\neq (0,0),
\end{equation}
 for all multi-indices $\beta,\gamma\in \na_0^n$ with
 $\abs{\beta+\gamma}\le 2n+1.$ In particular, these operators satisfy
 the weighted estimates of Theorem~\ref{thm:czweighted} and the
 variable Lebesgue space estimates in Corollary~\ref{coro:czvariable}.

 If \eqref{eq:CMcondition} is only satisfied for
 $\abs{\beta+\gamma}\le L,$ where $L<2n+1,$ then $T_\sigma$ may fail
 to be a bilinear Calder\'on-Zygmund operator. However, weighted
 estimates like those given in Theorem~\ref{thm:czweighted} do hold
 true for some bilinear multiplier operators with rougher symbols
 whose regularity is measured in terms of Sobolev norms. Consequently,
 such operators are also bounded in variable Lebesgue spaces. 

We state such results precisely.
Given $s,t\in\re,$  the Sobolev space $H^{(s,t)}(\rtn)$ consists of all $F\in \mathcal{S}'(\rtn)$ such that 
\[
\norm{F}{H^{(s,t)}}=\left(\int_{\rtn} (1+\abs{\tau_1}^2)^s(1+\abs{\tau_2}^2)^t |{\widehat{F}}(\tau_1,\tau_2)|^2\, d\tau_1 d\tau_2 \right)^{\frac{1}{2}}<\infty.
\]
 It follows  that $H^{s}(\rtn)\subset H^{(\frac{s}{2},\frac{s}{2})}(\rtn)$ for $s\ge 0,$  where $H^{s}(\rtn)$ is the Sobolev space consisting of $F\in L^2(\rtn)$ such that 
\[
\left(\int_{\rtn} (1+\abs{(\tau_1,\tau_2)}^2)^s |{\widehat{F}}(\tau_1,\tau_2)|^2\, d\tau_1 d\tau_2 \right)^{\frac{1}{2}}<\infty.
\]

Fix $\Psi\in\mathcal{S}(\rtn)$ such that
$\text{supp}(\Psi)\subset \{(\xi,\eta)\in\rtn:\frac{1}{2}\le
\abs{(\xi,\eta)}\le 2\}$
and $\sum_{k\in\ent}\Psi(2^{-k}\xi,2^{-k}\eta)=1$ for
$(\xi,\eta)\neq (0,0).$ Given a complex-valued, bounded function
$\sigma$ defined on $\rtn,$ set
$\sigma_k(\xi,\eta)=\sigma(2^k\xi,2^k\eta)\Psi(\xi,\eta)$ for
$\xi,\eta\in\rn$ and $k\in\ent.$ The following result from~\cite[Theorem 6.2]{MR2958938} is a weighted
version of a H\"ormander type theorem for bilinear Fourier
multipliers.
 
\begin{thm}\label{thm:fijutatomita} 
  Let $1<p,q<\infty,$ $\frac{1}{2}<r<\infty,$
  $\frac{1}{r}=\frac{1}{p}+\frac{1}{q}$ and $\frac{n}{2}<s,t\le n.$
  Assume $p>\frac{n}{s},$ $q>\frac{n}{t},$ $v\in A_{\frac{ps}{n}},$
  $w\in A_{\frac{qt}{n}}$ and $\sigma(\xi,\eta)$ is a complex-valued,
  bounded function defined for $\xi,\eta\in\rn$ such that
  $\sup_{k\in\ent}\norm{\sigma_k}{H^{(s,t)}}<\infty.$ Then for all
  $f\in L^p(v),\, g\in L^q(w)$,
\[
\norm{T_\sigma(f,g)}{L^r(v^{\frac{r}{p}}w^{\frac{r}{q}})}\lesssim \norm{f}{L^p(v)}\norm{g}{L^q(w)},
\]
where the implicit constant depends only on $p,$ $q,$
$[v]_{A_{\frac{ps}{n}}}$, $[w]_{A_{\frac{qt}{n}}}$  and $\sigma.$
\end{thm}

As an immediate consequence of Theorem~\ref{thm:fijutatomita} and
Theorem~\ref{thm:extrapolation} we get the boundedness
of bilinear multipliers in variable
Lebesgue spaces.

\begin{corollary} \label{cor:multiplier}
Given $\frac{n}{2}<s,t\le n,$ let
$p(\cdot),q(\cdot),r(\cdot)\in \Pp_0(\rn)$ satisfy
$\frac{1}{\rr}=\frac{1}{\pp}+\frac{1}{\qq}$, $p_{-}>\frac{n}{s},$ $q_{-}>\frac{n}{t},$
$r_{-}>\frac{1}{2}$  and assume there exist $\frac{n}{s}<p<p_{-}$ and $\frac{n}{t}<q<q_{-}$ such that
$\left(\frac{p(\cdot)}{p}\right)',\,\left(\frac{q(\cdot)}{q}\right)'\in\B.$
If $\sigma(\xi,\eta)$ is a complex-valued, bounded function defined
for $\xi,\eta\in\rn$ such that
$\sup_{k\in\ent}\norm{\sigma_k}{H^{(s,t)}}<\infty$,  then for all $f\in
L^{p(\cdot)}, \, g\in L^{q(\cdot)} $,
\begin{equation*}
\norm{T_{\sigma}(f,g)}{L^{r(\cdot)}} \lesssim \norm{f}{L^{p(\cdot)}}
\norm{g}{L^{q(\cdot)}}.
\end{equation*}
\end{corollary}

A slightly different version of Corollary~\ref{cor:multiplier} for the smaller class of symbols $\sigma$ satisfying $\sup_{k\in\ent}\norm{\sigma_k}{H^{s}(\rtn)}<\infty$ with $n< s\le 2n$ was
proved in~\cite{MR3359740}.   The proof again used linear
extrapolation and required the additional hypothesis that
$\frac{\rr}{r}\in \B$.

\begin{remark}
Weighted estimates like those in the hypothesis of
  Theorem~\ref{thm:extrapolation}, and the corresponding estimates on variable Lebesgue spaces, may be obtained for
  certain bilinear pseudodifferential operators $T_\sigma$, where
 $\sigma=\sigma(x,\xi,\eta)$ for $x,\xi,\eta\in\rn$.  We refer the
  reader to~\cite{MR3411149} and the references it contains for further
  details on these operators.
\end{remark}

\subsection*{Rough bilinear singular integrals in variable Lebesgue spaces.}

Let $\Omega \in L^\infty(\mathbb{S}^{2n-1})$ be such that
$\int_{\mathbb{S}^{2n-1}}\Omega \,d\sigma=0$ and define the bilinear
singular integral operator associated with $\Omega$ by
\begin{equation}\label{eq:Tomega}
T_\Omega(f,g)(x)=\text{p.v.}\int_{\rtn} K(x-y,x-z)f(y) g(z)\,dy dz,
\end{equation}
where $ f,g\in\sw, x\in\rn$ and
$K(y,z)=\frac{\Omega((y,z)/\abs{(y,z)})}{\abs{(y,z)}^{2n}}.$  These
operators were introduced by Coifman and Meyer~\cite{MR0380244}; for their
history see \cite{GHH2015} and the references it contains.   In this
latter paper it was
proved that $T_\Omega$ is bounded from $L^{p}\times L^q$ into $L^r$
for $1<p,q<\infty$ and $\frac{1}{r}=\frac{1}{p}+\frac{1}{q}.$  
In fact, we can readily adapt their proof to show that
$T_\Omega$ satisfies a weighted version of this result
when $1\le r<\infty.$

\begin{theorem}\label{thm:TOmegaweighted}
Let $\Omega \in L^\infty(\mathbb{S}^{2n-1})$ be such that
$\int_{\mathbb{S}^{2n-1}}\Omega \,d\sigma=0.$ If $1<p,q<\infty,$ $1\le
r<\infty,$ $\frac{1}{r}=\frac{1}{p}+\frac{1}{q},$ $v\in A_p$ and $w\in
A_q,$ then  for all $f\in L^p(v), \ g\in L^q(w)$,
 \begin{equation*}
\norm{T_\Omega (f,g)}{L^r(v^{\frac{r}{p}}w^{\frac{r}{q}})}\lesssim \norm{\Omega}{L^\infty} \norm{f}{L^p(v)}\norm{g}{L^q(w)},
\end{equation*}
where the implicit constant depends only on $p,$ $q,$
$[v]_{A_{p}}$, $[w]_{A_{q}}$ and $\Omega$.
\end{theorem}

\begin{proof} Let $p,q,r$ and $v,w$ be as in the hypotheses.  In
  \cite{GHH2015} they showed that $T_\Omega=\sum_{j\in\ent}T_j,$ where
  each $T_j$ is a Calder\'on-Zygmund operator.  Moreover, they proved
  \cite[Proposition 3 and Lemma 11]{GHH2015} that there exists $\delta>0$ such that
for all $f\in L^p, \, g\in L^q$,
\begin{equation}\label{eq:Tjunweighted}
\norm{T_j (f,g)}{L^r}\lesssim \norm{\Omega}{L^\infty} 2^{-\abs{j}\delta}\norm{f}{L^p}\norm{g}{L^q}.
\end{equation}

To get a similar weighted estimate, we adapt an interpolation argument
from~\cite{MR837527} to the bilinear setting.
In~\cite[Lemma 10]{GHH2015} they prove that  for
any $0<\varepsilon<1$, the kernel constant of
$T_j$ is controlled by
$C_{\varepsilon,n}\norm{\Omega}{L^\infty}2^{\varepsilon \abs{j}}$.  It
follows from the proof of Theorem~\ref{thm:czweighted} that the
constant in the weighted norm inequality depends linearly on
the kernel constant; hence, this gives us that
\begin{equation*}
\norm{T_j (f,g)}{L^r(v^{\frac{r}{p}}w^{\frac{r}{q}})}\lesssim C_{\varepsilon,n}\norm{\Omega}{L^\infty} 2^{\abs{j}\varepsilon}\norm{f}{L^p(v)}\norm{g}{L^q(w)}
\end{equation*}
for all  $f\in L^p(v),$ $g\in L^q(w)$ and $j\in\ent.$  By the sharp
reverse H\"older inequality~\cite{MR3092729}, there exists
  $0<\theta<1$ depending on $[v]_{A_p}$,  such that  $v^{\frac{1}{\theta}}\in A_p$ and
  $[v^{\frac{1}{\theta}}]_{A_p} \leq 2 [v]_{A_p}$; the same is true for
$w^{\frac{1}{\theta}}\in A_q$.    Therefore, the above argument in
fact implies that
\begin{equation}\label{eq:Tjweighted}
\norm{T_j (f,g)}{L^r(v^{\frac{r}{p\theta}}w^{\frac{r}{q\theta}})}
\lesssim C_{\varepsilon,n}\norm{\Omega}{L^\infty}
2^{\abs{j}\varepsilon}
\norm{f}{L^p(v^{\frac{1}{\theta}})}\norm{g}{L^q(w^{\frac{1}{\theta}})}
\end{equation}
for all  $f\in L^p(v^{\frac{1}{\theta}}),$ $g\in L^q(w^{\frac{1}{\theta}})$ and
$j\in\ent$, where the implicit constant only depends on $p$, $q$,
$[v]_{A_p}$ and $[w]_{A_q}$.  

Since we assumed that $1\le r<\infty,$ we can use complex interpolation with a
change of measure (\cite[Theorem 5.5.3]{MR0482275}) and bilinear
complex interpolation (\cite[Theorem 4.4.1]{MR0482275}) to get that
\eqref{eq:Tjunweighted} and \eqref{eq:Tjweighted} together imply
\begin{equation}\label{eq:Tjweightedimproved}
\norm{T_j (f,g)}{L^r(v^{\frac{r}{p}}w^{\frac{r}{q}})}\lesssim C_{\varepsilon,n}^{(1-\theta)}\norm{\Omega}{L^\infty} 2^{-\abs{j}((1-\theta)\delta-\theta\varepsilon)}\norm{f}{L^p(v)}\norm{g}{L^q(w)}
\end{equation}
for all $f\in L^p(v),g\in L^q(w)$ and $j\in\ent$. If we choose
$\varepsilon$ small enough that
$\rho=(1-\theta)\delta-\theta\varepsilon>0,$ we then obtain that
for all $f\in L^p(v),\, g\in L^q(w)$,
\begin{equation*}
\norm{T_\Omega (f,g)}{L^r(v^{\frac{r}{p}}w^{\frac{r}{q}})}\lesssim \norm{\Omega}{L^\infty} \sum_{j\in\ent }2^{-\abs{j}\rho}\norm{f}{L^p(v)}\norm{g}{L^q(w)}.
\end{equation*}
Since this series converges, we get the desired result.  
\end{proof}

Theorems~\ref{thm:extrapolation}  and \ref{thm:TOmegaweighted}
together immediately imply that $T_\Omega$ is bounded in the variable
Lebesgue spaces. 

\begin{corollary} \label{cor:TOmega-var}
Let
$p(\cdot),q(\cdot),r(\cdot)\in  \Pp(\rn)$ satisfy
$\frac{1}{\rr}=\frac{1}{\pp}+\frac{1}{\qq}$, $p_{-}>1,$ $q_{-}>1,$
$ r_{-}>1,$ and assume there exist $1<p<p_{-}$ and $1<q<q_{-}$ such that $ \frac{1}{p}+\frac{1}{q}\le 1$ and 
$\left(\frac{p(\cdot)}{p}\right)',\,\left(\frac{q(\cdot)}{q}\right)'\in\B.$
If $\Omega \in L^\infty(\mathbb{S}^{2n-1})$ with
$\int_{\mathbb{S}^{2n-1}}\Omega\,d\sigma=0$  then 
for all $f\in L^{p(\cdot)}, \, g\in L^{q(\cdot)}$, 
\begin{equation*}
\norm{T_\Omega(f,g)}{L^{r(\cdot)}} \lesssim \norm{f}{L^{p(\cdot)}}
\norm{g}{L^{q(\cdot)}}.
\end{equation*}
\end{corollary}

\begin{remark}
We conjecture that Theorem~\ref{thm:TOmegaweighted} is true when
$1/2<r<1$, which would in turn imply that
Corollary~\ref{cor:TOmega-var} holds when $1/2<r_-<1$.  In the proof
of the theorem, the argument holds when $r<1$ up through the proof
of~\eqref{eq:Tjweighted}.  It is only in applying complex
interpolation that we need the additional hypothesis that $r>1$.  
\end{remark}

\section{Proof of main results}\label{sec:proofmain}

In this section we prove Theorems~\ref{thm:KPweighted} and
\ref{thm:KPvariable}.  

\begin{proof}[Proof of Theorem~\ref{thm:KPweighted}]
  Our proof follows the broad outline of the proof
  of~\cite[Theorem~1]{MR3200091}, and we refer the reader there for
  some details.
Let $p,q,r, s$
  and $v,w$ be as in the hypotheses.
Fix a function $\f\in \sw$ such that $\supp(\f) \subset \{\xi \in \rn
:|\xi| \leq 2\}$ and $\f\equiv 1$ in $\{\xi\in\re^n:|\xi| \leq
1\}\label{eq:partunity1}$.  Define $\psi$ by $\psi(\xi)=\f(\xi) -
\f(2\xi)$ for  $\xi \in \rn$.  Then for all $k \in \ent$, 
\[ \supp(\psi(2^{-k}\cdot)) \subset \{\xi \in \rn: 2^{k-1}\leq |\xi|
\leq 2^{k+1}\}, \]
and for $\xi \neq 0$ and $M\in \ent$,
\[ \sum_{k\in\ent} \psi(2^{-k}\xi) \equiv \,1, \qquad \sum_{k\le M}
\psi(2^{-k}\xi) \equiv \,\f(2^{-M}\xi). \]

\smallskip

  We first prove the estimate \eqref{eq:KPweighted1}. 
Given $f,g\in\sw$, decompose
  $D^s(fg)$ as
\begin{align*}
&D^s(fg)(x) \\
& \qquad  = \int_{\rtn} \abs{\xi+\eta}^s \hf(\xi)\hg(\eta) \eixxe\dxe\\
&\qquad =\int_{\rtn} \abs{\xi+\eta}^s \left(\sum_{k\in\ent}\psi(2^{-k}\xi)\hf(\xi)\right)\left(\sum_{\ell\in\ent}\psi(2^{-\ell}\eta)\hg(\eta)\right)\eixxe\dxe\\
&\qquad =T_{1,s}(D^s f,g)(x) + T_{2,s}(f,D^s g)(x) + T_{3,s}(f,D^sg)(x),
\end{align*}
where 
\begin{align*}
T_{1,s}(f,g)&=\int_{\rtn}\Phi_1(\xi,\eta)\frac{\abs{\xi+\eta}^s}{\abs{\xi}^s}\hf(\xi)\hg(\eta)\eixxe\dxe,\\
T_{2,s}(f, g)&=\int_{\rtn}\Phi_2(\xi,\eta)\frac{\abs{\xi+\eta}^s}{\abs{\eta}^s}\hf(\xi)\hg(\eta)\eixxe\dxe,\\
T_{3,s}(f,g)&= \int_{\rtn}\Phi_3(\xi,\eta)\frac{\abs{\xi+\eta}^s}{\abs{\eta}^s}\hf(\xi)\hg(\eta)\eixxe\dxe,
\end{align*}
and
\begin{align*}
\Phi_1(\xi,\eta) & =\sum_{k\in\ent} \psi(2^{-k}\xi)\f(2^{-(k-5)}\eta), \\
\Phi_2(\xi,\eta) & =\sum_{k\in\ent} \f(2^{-(k-5)}\xi)\psi(2^{-k}\eta),\\
\Phi_3(\xi,\eta) & =\sum_{k\in\ent} \sum_{\ell=- 4}^4\psi(2^{-k}\xi)\psi(2^{-(k+\ell)}\eta).
\end{align*}

To complete the proof it will suffice to prove that for all $f,g\in\sw,$
\begin{align}
&\norm{T_{1,s}(f,g)}{L^r(v^{\frac{r}{p}}w^{\frac{r}{q}})}\lesssim \norm{f}{L^p(v)}\norm{g}{L^q(w)},\label{eq:KPweighted3}\\
&\norm{T_{2,s}(f,g)}{L^r(v^{\frac{r}{p}}w^{\frac{r}{q}})}\lesssim \norm{f}{L^p(v)}\norm{g}{L^q(w)},\label{eq:KPweighted4}\\
&\norm{T_{3,s}(f,g)}{L^r(v^{\frac{r}{p}}w^{\frac{r}{q}})}\lesssim \norm{f}{L^p(v)}\norm{g}{L^q(w)},\label{eq:KPweighted5}
\end{align}
The first two inequalities are straightforward.
Since $\Phi_1$ is supported in
$\{(\xi,\eta)\in\rtn:\abs{\eta}\le \frac{1}{8}\abs{\xi}\}$ and
$\Phi_2$ is supported in
$\{(\xi,\eta)\in\rtn:\abs{\xi}\le \frac{1}{8}\abs{\eta}\},$ it follows
that $T_{1,s}$ and $T_{2,s}$ are bilinear Coifman-Meyer multiplier
operators; therefore, by Theorem~\ref{thm:czweighted} we have that
\eqref{eq:KPweighted3} and \eqref{eq:KPweighted4} hold.  

If $s$ is a non-negative even integer or $s$ is sufficiently large,
them $T_{3,s}$ is also a Coifman-Meyer multiplier operator and so 
\eqref{eq:KPweighted5} holds in these cases. Otherwise, $T_{3,s}$ may
fail to be a Coifman-Meyer multiplier operator (see \cite[Remark 1,
p. 1139]{MR3200091}). In general, however, as shown in \cite[p. 1137]{MR3200091},
$T_{3,s}$ can be written as a finite sum of terms (one for
each value of $\ell$) of the form 
\begin{equation*}
\sum_{m \in \zz^n} c_{s,m} T_m(f,g),
\end{equation*}
with  $\abs{c_{m,s}} \lesssim (1+|m|)^{-(n+s)}$  and 
\begin{equation*}
T_m(f,g)(x) = \sum_{k \in \ent} (\Psi^1_{k,m}*f)(x)({\Psi}^2_{k,m}*g)(x),
\end{equation*}
where for $j=1,2,$ $k\in\ent$ and $m\in\ent^n,$
$\Psi^j_{k,m}(x)= 2^{kn} \Psi^j(2^k x + m)$ for some smooth function
$\Psi^j$ such that $\widehat{\Psi^j}$ is supported in an annulus. By
the Cauchy-Schwarz inequality,
\begin{multline*}
\abs{\sum_{m \in \zz^n} c_{s,m} T_m(f,g)(x)} \\
\leq  \sum_{m \in \zz^n} (1+|m|)^{-(n+s)} \left(\sum_{k \in \ent} |\Psi^1_{k,m}*f(x)|^2 \right)^{1/2}  \left(\sum_{k \in \ent} |\Psi^2_{k,m}*g(x)|^2 \right)^{1/2}.
\end{multline*}

Define the square function $S_m^j$, $j=1,\,2$, by 
\begin{align*}
S_m^j(f)(x)= \left(\sum_{k \in \ent} |\Psi^j_{k,m}*f(x)|^2 \right)^{1/2}.
\end{align*}
By Theorem~\ref{thm:sqweighted} the operators $S^j_m$ satisfy weighted
estimates that are uniform in $m;$ therefore, with $r_*=\min(r,1),$ we
get that for all $f,g\in\sw$,
\begin{multline*}
\norm{\sum_{m \in \zz^n} c_{s,m}
  T_m(f,g)}{L^{r}(v^{\frac{r}{p}}w^{\frac{r}{q}})}^{r_*}
\lesssim \sum_{m \in \zz^n} (1+|m|)^{-(n+s)r_*} \norm{S_m^1 f}{L^{p}(v)}^{r_*}\norm{S_m^2 g}{L^{q}(w)}^{r_*}\\
\lesssim \sum_{m \in \zz^n} (1+|m|)^{-(n+s)r_*}\norm{f}{L^{p}(v)}^{r_*}\norm{g}{L^{q}(w)}^{r_*}.
\end{multline*}
By assumption, $s>\max\{0,n(\frac{1}{r}-1)\}$;  hence, $(n+s)r_*>n$
and so the series in $m$ converges.   This proves
inequality~\eqref{eq:KPweighted5} and so completes the proof of \eqref{eq:KPweighted1}.

\bigskip

The proof of \eqref{eq:KPweighted2} is similar.  We can  decompose $J^s(f,g)$ as
\[
J^s(f,g)(x)={\widetilde T}_{1,s}(J^sf,g)+{\widetilde T}_{2,s}(f,J^sg)+{\widetilde T}_{3,s}(f,J^sg);
\]
the operators ${\widetilde T}_{1,s}$ and ${\widetilde T}_{2,s}$ are
defined like $T_{1,s}$ and $T_{2,s}$ with
$\frac{\abs{\xi+\eta}^s}{\abs{\xi}^s}$ and
$\frac{\abs{\xi+\eta}^s}{\abs{\eta}^s}$ replaced by
$\frac{(1+\abs{\xi+\eta}^2)^\frac{s}{2}}{(1+\abs{\xi}^2)^{\frac{s}{2}}}$
and
$\frac{(1+\abs{\xi+\eta}^2)^\frac{s}{2}}{(1+\abs{\eta}^2)^{\frac{s}{2}}}$
respectively.  Again,  both ${\widetilde T}_{1,s}$ and ${\widetilde T}_{2,s}$
are bilinear Coifman-Meyer multiplier operators; therefore, weighted
estimates corresponding to \eqref{eq:KPweighted3} and \eqref{eq:KPweighted4}
hold for ${\widetilde T}_{j,s}$ for $j=1,2,$,
respectively. 

To estimate $\widetilde{T}_{3,s}$, we use the fact (see
\cite[pp. 1148-1149]{MR3200091}) that it can be written as
\[
{\widetilde T}_{3,s}(f,g)(x)={\widetilde T}_{3,s}^1(f,g)(x)+{\widetilde T}_{3,s}^2(f,g)(x),\quad x\in\rn,
\]
where 
\begin{align*}
\abs{{\widetilde T}_{3,s}^1(f,g)(x)}\le \sum_{m,\ell\in\ent^n} b_{m}^s\tilde{b}_\ell^s \sum_{k\ge 0} \abs{(\Psi_{k,m}\ast f)(x)( \Psi_{k,m+\ell}\ast g)(x)}
\end{align*}
and
\begin{align*}
{\widetilde T}_{3,s}^2(f,g)(x)= \sum_{m,\ell\in\ent^n}
  a_{m}^s\tilde{a}_\ell^s \sum_{k< 0} (\widetilde{\Psi}_{k,m} \ast f)( \widetilde{\Psi}_{k,m+\ell}\ast g),
\end{align*}
where $\Psi_{k,m}(x)=2^{kn}\Psi(2^k x+m)$ and
$\widetilde{\Psi}_{k,m}(x)=2^{kn}\Psi(2^k (x-m))$ with $\Psi\in \sw$
and $\widehat{\Psi}$ supported in an annulus. The constants satisfy 
\[ \abs{b_{m}^s}, \, |\tilde{b}_m^s|, \,\abs{a_{m}^s}, \, \abs{\tilde{a}_m^s}
\lesssim \abs{m}^{-n-s}.  \]
We can now argue as we did for $T_{3,s}$ above and use
Theorem~\ref{thm:sqweighted}  to get the corresponding estimate
\eqref{eq:KPweighted5} for ${\widetilde T}_{3,s}$.

\bigskip

We next prove \eqref{eq:Commweighted1}. The decompositions given below are inspired by those used in  the proof of \cite[Lemma X1]{MR951744}  corresponding to $J^s.$ We have
\begin{multline*}
D^s(fg)(x) -fD^s(g)(x) = \int_{\rtn} (\abs{\xi+\eta}^s-\abs{\eta}^s)
\hf(\xi)\hg(\eta) \eixxe\dxe\\ 
=Q_{1,s}(f,g)(x) + Q_{2,s}(f,g)(x) + Q_{3,s}(f,g)(x),
\end{multline*}
where 
\begin{align*}
Q_{1,s}(f,g)&=\int_{\rtn}\Phi_1(\xi,\eta)(\abs{\xi+\eta}^s-\abs{\eta}^s)\hf(\xi)\hg(\eta)\eixxe\dxe,\\
Q_{2,s}(f, g)&=\int_{\rtn}\Phi_2(\xi,\eta)(\abs{\xi+\eta}^s-\abs{\eta}^s)\hf(\xi)\hg(\eta)\eixxe\dxe,\\
Q_{3,s}(f,g)&= \int_{\rtn}\Phi_3(\xi,\eta)(\abs{\xi+\eta}^s-\abs{\eta}^s)\hf(\xi)\hg(\eta)\eixxe\dxe.
\end{align*}

The operator $Q_{1,s}$ can in turn be decomposed  as 
\begin{align*}
Q_{1,s}(f,g)(x)=Q_{1,s}^1(D^sf,g)(x)-Q_{1,s}^2(f,g)(x),
\end{align*}
where $Q_{1,s}^1$ is the same as $T_{1,s}$  (and hence satisfies \eqref{eq:KPweighted3}) and 
\[
Q_{1,s}^2(f,g)(x)=\int_{\rtn} \Phi_1(\xi,\eta)\abs{\eta}^s\widehat{f}(\xi) \widehat{g}(\eta)\eixxe\dxe.
\]
Moreover,
\begin{multline*}
\Phi_1(\xi,\eta)\abs{\eta}^s\widehat{f}(\xi) \widehat{g}(\eta)
=\Phi_1(\xi,\eta)\abs{2\pi \xi}^{-2} \abs{2\pi\xi}^2 \widehat{f}(\xi)
  \abs{\eta}^{-1}\abs{\eta}^2\widehat{D^{s-1}g}(\eta)\\ 
=\frac{1}{2\pi}\Phi_1(\xi,\eta)\abs{\xi}^{-2} \sum_{j=1}^n\xi_j \widehat{\partial_jf}(\xi) \sum_{k=1}^n\eta_k\widehat{G_kD^{s-1}g}(\eta),
\end{multline*}
where $\widehat{G_k
  h}(\eta)=\frac{\eta_k}{\abs{\eta}}\widehat{h}(\eta)$ is a constant
multiple of the Riesz transform $R_k$.  
Therefore, we have that
\[
Q_{1,s}^2(f,g)(x)=\sum_{j,k=1}^n Q_{1,s}^{2,j,k}(\partial_j f, G_kD^{s-1}g),
\]
where $Q_{1,s}^{2,j,k}$ is a bilinear multiplier operator with symbol
$\frac{1}{2\pi}\xi_j\eta_k\abs{\xi}^{-2}\Phi_1(\xi,\eta).$ Such a symbol
is a Coifman-Meyer multiplier since $\Phi_1$ is supported in the set
$\{(\xi,\eta)\in\rtn:\abs{\eta}\le \frac{1}{8}\abs{\xi}\};$ therefore, 
we can apply Theorem~\ref{thm:czweighted} to $Q_{1,s}^{2,j,k}$ 
to get that for all $ f,g\in\sw$,
\begin{equation*}
\norm{Q_{1,s}^2(f,g)}{L^r(v^{\frac{r}{p}}w^{\frac{r}{q}})}\lesssim \norm{\nabla f}{L^{p}(v)}\norm{D^{s-1}g}{L^q(w)}.
\end{equation*}
where we have used that $G_k$ is a bounded operator from $L^q(w)$ into
$L^q(w)$, since the Riesz transforms are (see~\cite{MR1800316}). 
If we combine these estimates, we see that for all $f,g\in\sw$,
\begin{equation}\label{eq:Commweighted3}
\norm{Q_{1,s}(f,g)}{L^r(v^{\frac{r}{p}}w^{\frac{r}{q}})}\lesssim \norm{D^sf}{L^p(v)}\norm{g}{L^q(w)}+ \norm{\nabla f}{L^{p}(v)}\norm{D^{s-1}g}{L^q(w)}.
\end{equation}

The symbol of the operator $Q_{2,s}$ is given by
\begin{align*}
\Phi_2(\xi,\eta) (\abs{\xi+\eta}^s-\abs{\eta}^s)&=\Phi_2(\xi,\eta) \abs{\eta}^s\left(\frac{\abs{\xi+\eta}^s}{\abs{\eta}^s}-1\right)\\
&=\Phi_2(\xi,\eta) \abs{\eta}^s\left[\left(1+\frac{\abs{\xi}^2+2\xi\cdot\eta}{\abs{\eta}^2}\right)^{s/2}-1\right]\\
&=\Phi_2(\xi,\eta) \abs{\eta}^s\sum_{j=1}^\infty {{s/2} \choose j} \left(\frac{\abs{\xi}^2+2\xi\cdot\eta}{\abs{\eta}^2}\right)^{j},
\end{align*}
where ${s/2\choose j}=\frac{s/2(s/2-1)(s/2-2)\cdots (s/2-j+1)}{j!}$
and the series converges absolutely and uniformly on the support of
$\Phi_2.$ Indeed, the support of $\Phi_2$ is contained in
$\{(\xi,\eta)\in\rtn: \abs{\xi}\le \frac{1}{8}\abs{\eta}\}$ and for
$(\xi,\eta)$ in this set we have
$\abs{\frac{\abs{\xi}^2+2\xi\cdot\eta}{\abs{\eta}^2}}\le
\frac{17}{64}$. 
(Recall that the radius of convergence of the binomial series is 1 for
exponents that are not in $\na_0$; otherwise the sum is finite.) Now,
with $c_{j,s}={{s/2} \choose j},$
\begin{multline*}
c_{j,s}\Phi_2(\xi,\eta) \abs{\eta}^s\left(\frac{\abs{\xi}^2+2\xi\cdot\eta}{\abs{\eta}^2}\right)^{j}\widehat{f}(\xi)\widehat{g}(\eta)\\
=\frac{c_{j,s}}{2\pi}\sum_{\nu=1}^n\Phi_2(\xi,\eta)\frac{(\abs{\xi}^2+2\xi\cdot\eta)^{j-1}}{\abs{\eta}^{2j-1}} (\xi_\nu+2\eta_\nu)\widehat{\partial_\nu f}(\xi)\widehat{D^{s-1}g}(\eta);
\end{multline*}
setting $\sigma_{j,\nu}(\xi,\eta)=\Phi_2(\xi,\eta)\frac{(\abs{\xi}^2+2\xi\cdot\eta)^{j-1}}{\abs{\eta}^{2j-1}} (\xi_\nu+2\eta_\nu),$  we have that 
\[
Q_{2,s}(f,g)(x)=\frac{1}{2\pi}\sum_{j=1}^\infty  \sum_{\nu=1}^n c_{j,s} T_{\sigma_{j,\nu}}(\partial_\nu f, D^{s-1}g)(x)\quad \forall f,g\in\sw, x\in\rn,
\]
where $T_{\sigma_{j,\nu}}$ is the bilinear multiplier operator with
symbol $\sigma_{j,\nu}.$ By Lemma~\ref{lemma:sigmajnu} (whose
statement and proof we defer to the end of this section),
$\sigma_{j,\nu}$ is a Coifman-Meyer multiplier for each $j\in\na$ and
$\nu=1,\ldots,n$.  Further, 
$\sum_{j=1}^\infty \abs{c_{j,s}}
\|T_{\sigma_{j,\nu}}\|_{p,q,v,w}<\infty,$
where $\|T_{\sigma_{j,\nu}}\|_{p,q,v,w}$ is the norm of
$T_{\sigma_{j,\nu}}$ as a bounded operator from $L^p(v)\times L^q(w)$
into $L^r(v^{\frac{r}{p}}w^{\frac{r}{q}}).$ This implies that  for $f,\, g\in\sw$, 
\begin{equation}\label{eq:Commweighted4}
\norm{Q_{2,s}(f,g)}{L^r(v^{\frac{r}{p}}w^{\frac{r}{q}})}\lesssim \norm{\nabla f}{L^{p}(v)}\norm{D^{s-1}g}{L^q(w)}.
\end{equation}

Finally, the operator $Q_{3,s}$ can be written as 
\[
Q_{3,s}(f,g)(x)= Q_{3,s}^1(D^sf,g)(x)-Q_{3,s}^2(D^sf,g)(x),
\]
where 
\begin{align*}
&Q_{3,s}^1(f,g)(x)=\int_{\rtn}\Phi_3(\xi,\eta) \frac{\abs{\xi+\eta}^s}{\abs{\xi}^s}\widehat{f}(\xi)\widehat{g}(\eta)\eixxe\dxe,\\
&Q_{3,s}^2(f,g)(x)=\int_{\rtn}\Phi_3(\xi,\eta) \frac{\abs{\eta}^s}{\abs{\xi}^s}\widehat{f}(\xi)\widehat{g}(\eta)\eixxe\dxe.
\end{align*}
The operator $Q^1_{3,s}$ is the same as the operator $T_{3,s}$ but
with the roles of $f$ and $g$ interchanged. The operator 
 $Q_{3,s}^2$ is a Coifman-Meyer multiplier operator since
$\Phi_3$ is supported in a region where $\abs{\xi}\sim
\abs{\eta}$.
Therefore, for $j=1,\,2$,  $Q_{3,s}^j(f,g)$ satisfies the weighted
estimates in the conclusion of 
Theorem~\ref{thm:czweighted}, and so we have that for all $f,g\in\sw$,
\begin{equation}\label{eq:Commweighted5}
\norm{Q_{3,s}(f,g)}{L^r(v^{\frac{r}{p}}w^{\frac{r}{q}})}\lesssim \norm{D^sf}{L^p(v)}\norm{g}{L^q(w)}.
\end{equation}
The estimate \eqref{eq:Commweighted1} now follows from
\eqref{eq:Commweighted3}, \eqref{eq:Commweighted4} and
\eqref{eq:Commweighted5}.

\medskip

Finally to prove \eqref{eq:Commweighted2} we describe the changes we
need to make in the preceding argument
(see \cite[Lemma X1]{MR951744}). In this
case, the operators $Q_{j,s}$ are replaced by 
the operators $\widetilde{Q}_{j,s}$, $j=1,2,3$, which in their
symbols have
$(1+\abs{\xi+\eta}^2)^\frac{s}{2}-(1+\abs{\eta}^2)^\frac{s}{2}$
instead of $\abs{\xi+\eta}^s-\abs{\eta}^s.$  The operator
$\widetilde{Q}_{1,s}$ can be decomposed as
 \begin{align*}
\widetilde{Q}_{1,s}(f,g)(x)=\widetilde{Q}_{1,s}^1(J^sf,g)(x)-\widetilde{Q}_{1,s}^2(f,g)(x),
\end{align*} 
where $\widetilde{Q}_{1,s}^1$ has symbol
$\Phi_1(\xi,\eta)\frac{(1+\abs{\xi+\eta}^2)^{\frac{s}{2}}-1}{(1+\abs{\xi}^2)^{\frac{s}{2}}},$
which is a Coifman-Meyer multiplier, and $\widetilde{Q}_{1,s}^2$ has
symbol $\Phi_1(\xi,\eta)((1+\abs{\eta}^2)^{\frac{s}{2}}-1).$ We then
get
 \[
\widetilde{Q}_{1,s}^2(f,g)(x)=\sum_{j,k=1}^n \widetilde{Q}_{1,s}^{2,j,k}(\partial_j f, \widetilde{G}_kJ^{s-1}g)
\]
where $\widetilde{Q}_{1,s}^{2,j,k}$ is a bilinear multiplier operator
with symbol $\frac{1}{2\pi}\xi_j\eta_k\abs{\xi}^{-2}\Phi_1(\xi,\eta)$
(i.e. $\widetilde{Q}_{1,s}^{2,j,k}=Q_{1,s}^{2,j,k}$) and
$\widetilde{G}_k$ is the linear multiplier operator given by
\[
\widehat{\widetilde{G}_kh}(\eta)=\frac{\eta_k(1+\abs{\eta}^2)^{\frac{1-s}{2}}((1+\abs{\eta}^2)^{\frac{s}{2}}-1)}{\abs{\eta}^2}\,\widehat{h}(\eta).
\]
By the weighted version of the H\"ormander-Mihlin theorem
(see~\cite{MR542885}), $\widetilde{G}_k$ is a bounded operator in
$L^q(w).$   As a consequence of these arguments we get \eqref{eq:Commweighted3} with
$Q_{1,s}$ and $D$ replaced with $\widetilde{Q}_{1,s}$ and $J,$
respectively.

For the operator $\widetilde{Q}_{2,s},$ we have that its symbol satisfies
\begin{multline*}
\Phi_2(\xi,\eta) \left((1+\abs{\xi+\eta}^2)^\frac{s}{2}-(1+\abs{\eta}^2)^\frac{s}{2}\right)\\
=\Phi_2(\xi,\eta) (1+\abs{\eta}^2)^\frac{s}{2}\sum_{j=1}^\infty {{s/2}
  \choose j}
\left(\frac{\abs{\xi}^2+2\xi\cdot\eta}{1+\abs{\eta}^2}\right)^{j}, 
\end{multline*}
with the series converging uniformly and absolutely on the support of
$\Phi_2,$ since
$\abs{\frac{\abs{\xi}^2+2\xi\cdot\eta}{1+\abs{\eta}^2}}\le
\frac{17}{64}<1$ if $(\xi,\eta)\in \supp(\Phi_2)$.  Hence, we have
that for all $f,\, g\in\sw$ and $x\in\rn$,
\[
\widetilde{Q}_{2,s}(f,g)(x)=\frac{1}{2\pi}\sum_{j=1}^\infty  \sum_{\nu=1}^n {{s/2} \choose j} T_{\widetilde{\sigma}_{j,\nu}}(\partial_\nu f, J^{s-1}g)(x),
\]
where 
\[
\widetilde{\sigma}_{j,\nu}(\xi,\eta)=\Phi_2(\xi,\eta)\frac{(\abs{\xi}^2+2\xi\cdot\eta)^{j-1}}{(1+\abs{\eta}^2)^{j-\frac{1}{2}}} (\xi_\nu+2\eta_\nu).
\]
The counterpart of \eqref{eq:Commweighted4} for
$\widetilde{Q}_{2,s}$ (with $D$ replaced by $J$) now follows from
Lemma~\ref{lemma:sigmajnu}.

Finally, the operator $\widetilde{Q}_{3,s}$ is treated in the same way as
$Q_{3,s}.$ The corresponding operator $\widetilde{Q}_{3,s}^1$ is
analogous to $\widetilde{T}_{3,s},$ while $\widetilde{Q}_{3,s}^2$ has
symbol
$\Phi_3(\xi,\eta)\left(\frac{1+\abs{\eta}^2}{1+\abs{\xi}^2}\right)^\frac{s}{2},$
which is a Coifman-Meyer multiplier. Inequality \eqref{eq:Commweighted5}
with $Q_{3,s}$ and $D$ replaced by $\widetilde{Q}_{3,s}$ and $J,$
respectively, follows as before.
\end{proof}

\bigskip

\begin{proof}[Proof of Theorem~\ref{thm:KPvariable}]
The desired inequalities all follow at once from bilinear
extrapolation and the weighted estimates derived in the proof of
Theorem~\ref{thm:KPweighted}.    To prove \eqref{eq:KPvariable1} and
\eqref{eq:KPvariable2}, it suffices to note that  the weighted
  inequalities proved for the operators $T_{j,s}$,  $j=1,2,3$---i.e.,
  \eqref{eq:KPweighted3}, \eqref{eq:KPweighted4} and
  \eqref{eq:KPweighted5}---satisfy the hypotheses of
  Theorem~\ref{thm:extrapolation}, and so we get the corresponding
  variable Lebesgue space estimates.  The same is true for the operators
  $\widetilde{T}_{j,s}$,  $j=1,2,3$.

To prove \eqref{eq:Commvariable1} and \eqref{eq:Commvariable2}, we
again note that the weighted norm inequalities for  $Q_{j,s}$,
$j=1,2,3$---i.e., \eqref{eq:Commweighted3}, \eqref{eq:Commweighted4}
and \eqref{eq:Commweighted5}---again satisfy the hypotheses of
  Theorem~\ref{thm:extrapolation}, and so we get the corresponding
  variable Lebesgue space estimates.  The same is true for the
  operators
$\widetilde{Q}_{j,s}$, $j=1,2,3.$
\end{proof}

We conclude this section by stating and proving the lemma used in the
proof of Theorem~\ref{thm:KPweighted}.  

\begin{lemma}\label{lemma:sigmajnu} For $j\in\na$ and $\nu=1,\ldots,
  n,$
  let $\sigma_{j,\nu}$ and $\widetilde{\sigma}_{j,\nu}$ be as in the
  proof of Theorem~\ref{thm:KPweighted}, let $p,q,r,v,w$ be as in the
  hypotheses of Theorem~\ref{thm:KPweighted},  and  fix $s>0.$ Then
  $T_{\sigma_{j,\nu}}$ is a Coifman-Meyer multiplier operator and
  $\sum_{j=1}^\infty |{s/2 \choose j}|
  \|T_{\sigma_{j,\nu}}\|_{p,q,v,w}<\infty,$
  where $\|T_{\sigma_{j,\nu}}\|_{p,q,v,w}$ is the norm of
  $T_{\sigma_{j,\nu}}$ as a bounded operator from
  $L^p(v)\times L^q(w)$ into $L^r(v^{\frac{r}{p}}w^{\frac{r}{q}})$. The analogous
 result also holds for $\widetilde{\sigma}_{j,\nu}.$
\end{lemma}

\begin{proof} 
  We will prove Lemma~\ref{lemma:sigmajnu} using
  Theorem~\ref{thm:czweighted} and the fact that for Coifman-Meyer
  multipliers, the constant in the weighted norm inequality depends
  linearly on the constant in~\eqref{eq:CMcondition}.  Recall that
  $\sigma_{j,\nu}=\frac{(\abs{\xi}^2+2\xi\cdot\eta)^{j-1}}{\abs{\eta}^{2j-1}}
  (\xi_\nu+2\eta_\nu) \Phi_2(\xi,\eta),$
  where $\Phi_2$ is supported in the set
  $\{(\xi,\eta)\in\rtn: \abs{\xi}\le \frac{1}{8}\abs{\eta}\}$ and $\xi_\nu, \eta_\nu$ denote the $\nu$-th coordinates of $\xi$ and $\eta,$ respectively. Set
  $\tau_{j,\nu}(\xi,\eta)=\frac{(\abs{\xi}^2+2\xi\cdot\eta)^{j-1}}{\abs{\eta}^{2j-1}}
  (\xi_\nu+2\eta_\nu)$
  and note that $\tau_{j,\nu}$ is homogeneous of degree 0 in its
  domain.  For $\beta,\gamma\in\na_0^n,$ we then have that for all
  $ (\xi,\eta), \eta\neq 0$,
\begin{equation*}
\partial_\xi^\beta\partial_\eta^\gamma\tau_{j,\nu}(\xi,\eta)=\abs{(\xi,\eta)}^{-\abs{\beta+\gamma}}(\partial_\xi^\beta\partial_\eta^\gamma\tau_{j,\nu})(\fr{(\xi,\eta)}{\abs{(\xi,\eta)}});
\end{equation*}
therefore, for all $(\xi,\eta)\in \supp(\Phi_2)\setminus\{(0,0)\}$,
\begin{equation*}
\abs{\partial_\xi^\beta\partial_\eta^\gamma\tau_{j,\nu}(\xi,\eta)}\lesssim
\sup_{(a,b)\in\text{supp}(\Phi_2)\setminus\{(0,0)\}}
\abs{(\partial_a^\beta\partial_b^\gamma\tau_{j,\nu})(\fr{(a,b)}{\abs{(a,b)}})}
\abs{(\xi,\eta)}^{-\abs{\beta+\gamma}}. 
\end{equation*}
Note that
$(\partial_a^\beta\partial_b^\gamma\tau_{j,\nu})(\fr{(a,b)}{\abs{(a,b)}})$
is bounded on the set $\supp(\Phi_2)\setminus \{(0,0)\}$ since
$\frac{8}{\sqrt{65}}\le \frac{\abs{b}}{\abs{(a,b)}}\le 1$ for
$(a,b)\in \supp(\Phi_2)\setminus \{(0,0)\}$.  By the product rule for
derivatives, the above inequality combined with the fact that $\Phi_2$
is a Coifman-Meyer multiplier implies that $\sigma_{j,\nu}$ is also a
Coifman-Meyer multiplier.  Moreover, for all $(\xi,\eta)\neq
(0,0)$, 
\begin{equation*}
\abs{\partial_\xi^\beta\partial_\eta^\gamma
  \sigma_{j,\nu}(\xi,\eta)}\lesssim
\mathop{\sup_{(a,b)\in\text{supp}(\Phi_2)\setminus\{(0,0)\}}}_{\tilde{\beta}\le
  \beta,\tilde{\gamma}\le
  \gamma}\abs{(\partial_a^{\tilde{\beta}}\partial_b^{\tilde{\gamma}}\tau_{j,\nu})(\fr{(a,b)}{\abs{(a,b)}})}
\abs{(\xi,\eta)}^{-\abs{\beta+\gamma}},
\end{equation*}
where the implicit constant is independent of $j.$ By
Theorem~\ref{thm:czweighted} we have that 
\[
\|T_{\sigma_{j,\nu}}\|_{p,q,v,w}\lesssim \sup_{\abs{\beta+\gamma}\le 2n+1}\sup_{(\xi,\eta)\neq (0,0)}\abs{\partial_\xi^\beta\partial_\eta^\gamma \sigma_{j,\nu}(\xi,\eta)} \abs{(\xi,\eta)}^{\abs{\beta+\gamma}},
\]
where the implicit constant is again independent of $j.$ Together,
these two inequalities imply that for all $j$, 
\[
\|T_{\sigma_{j,\nu}}\|_{p,q,v,w}\lesssim \sup_{\abs{\beta+\gamma}\le 2n+1}\sup_{(\xi,\eta)\in S} \abs{(\partial_\xi^\beta\partial_\eta^\gamma\tau_{j,\nu})(\xi,\eta)},
\]
where $S=\mathbb{S}^{2n-1}\cap \{(\xi,\eta):\abs{\xi}\le \frac{1}{8}\abs{\eta}\}.$ 

Set $F(\xi,\eta)=\frac{\xi_\nu+2\eta_\nu}{\abs{\eta}}$ and 
$G(\xi,\eta)=\frac{\abs{\xi}^2+2\xi\cdot \eta}{\abs{\eta}^2}$; then
$
\tau_{j,\nu}(\xi,\eta)=F(\xi,\eta)\left(G(\xi,\eta)\right)^{j-1}.
$
By induction we have that if $j>2n+2$  and $\abs{\beta+\gamma}\le 2n+1$, then 
\begin{multline}
\partial_\xi^\beta\partial_\eta^\gamma\tau_{j,\nu}(\xi,\eta)=\partial_{\xi}^\beta\partial_\eta^\gamma F(\xi,\eta) (G(\xi,\eta))^{j-1}\label{eq:sigmaFG}\\
+\mathop{\mathop{\sum_{\beta_1+\beta_2=\beta}}_{\gamma_1+\gamma_2=\gamma}}_{(\beta_2,\gamma_2)\neq
  (0,0)} \partial_{\xi}^{\beta_1}\partial_\eta^{\gamma_1} F(\xi,\eta)
  \sum_{\ell=2}^{\abs{\beta_2+\gamma_2}+1}
  \prod_{\kappa=1}^{\ell-1}(j-\kappa)
(G(\xi,\eta))^{j-\ell} H_{\beta_2,\gamma_2}(\xi,\eta),
\end{multline}
where $H_{\beta_2,\gamma_2}$ is a sum of products of derivatives of
$G$ with each product having $\beta_2$ derivatives with respect to
$\xi$ and $\gamma_2$ derivatives with respect to $\eta.$ The
derivatives of $F$ and $G$ are bounded on $S$ and
$\abs{G(\xi,\eta)}\le \frac{17}{64}$ for $(\xi,\eta)\in S.$ Therefore,
for all $j>2n+2$ we have that
\begin{multline*}
\sup_{\abs{\beta+\gamma}\le 2n+1}\sup_{(\xi,\eta)\in
  S}\abs{\partial_\xi^\beta\partial_\eta^\gamma\tau_{j,\nu}(\xi,\eta)}
\lesssim
\left(\fr{17}{64}\right)^{j-1}+\sum_{\ell=2}^{2n+2} (j-1)\cdots
(j-\ell+1)\left(\fr{17}{64}\right)^{j-\ell}\\ 
\lesssim \left(\fr{17}{64}\right)^{j-1}+ \left(\fr{17}{64}\right)^{j} \sum_{\ell=2}^{2n+2} {{j-1}\choose \ell-1}\lesssim \left(\fr{17}{64}\right)^{j}+ \left(\fr{17}{32}\right)^{j}.
\end{multline*}
Thus,
\[
\sum_{j=2n+3}^\infty {\textstyle |{s/2\choose j}|} \|T_{\sigma_{j,\nu}}\|_{p,q,v,w}\lesssim \sum_{j=2n+3}^\infty {\textstyle |{s/2\choose j}|} \big[\left(\fr{17}{64}\right)^{j}+ \left(\fr{17}{32}\right)^{j}\big] <\infty,
\] 
which implies the desired result for $\sigma_{j,s}.$

\bigskip

To prove the analogous result for $\widetilde{\sigma}_{j,\nu}$, recall that
\[ \widetilde{\sigma}_{j,\nu}(\xi,\eta)=\Phi_2(\xi,\eta)\frac{(\abs{\xi}^2+2\xi\cdot\eta)^{j-1}}{(1+\abs{\eta}^2)^{j-\frac{1}{2}}}
(\xi_\nu+2\eta_\nu) =
\Phi_2(\xi,\eta)\widetilde{\tau}_{j,\nu}(\xi,\eta). \]
The product rule and the fact that $\Phi_2$ is a Coifman-Meyer
multiplier imply that
\begin{equation*}
\abs{\partial_\xi^\beta\partial_\eta^\gamma \widetilde{\sigma}_{j,\nu}(\xi,\eta)}\lesssim \mathop{\sup_{(a,b)\in\text{supp}(\Phi_2)\setminus\{(0,0)\}}}_{\tilde{\beta}\le \beta,\tilde{\gamma}\le \gamma}\abs{\abs{(a,b)}^{\abs{\tilde{\beta}+\tilde{\gamma}}}
(\partial_a^{\tilde{\beta}}\partial_b^{\tilde{\gamma}}
\widetilde{\tau}_{j,\nu})(a,b)} \abs{(\xi,\eta)}^{-\abs{\beta+\gamma}}
\end{equation*}
for all $(\xi,\eta)\neq (0,0)$ and  all  $j.$ This computation is similar to the estimate for $\tau_{j,\nu}$ above, which only uses homogeneity to rescale the derivative.  The boundedness of the supremum is a consequence of the argument below. Thus for all $j$ we have that
\begin{multline*}
  \|T_{\widetilde{\sigma}_{j,\nu}}\|_{p,q,v,w}
  \lesssim \sup_{\abs{\beta+\gamma}\le 2n+1}\sup_{(\xi,\eta)\neq (0,0)}\abs{\partial_\xi^\beta\partial_\eta^\gamma \widetilde{\sigma}_{j,\nu}(\xi,\eta)} \abs{(\xi,\eta)}^{\abs{\beta+\gamma}}\\
\lesssim
  \mathop{\sup_{(\xi,\eta)\in\text{supp}(\Phi_2)\setminus\{(0,0)\}}}_{\abs{\beta+\gamma}\le
    2n+1}\abs{\abs{(\xi,\eta)}^{\abs{\beta+\gamma}}
    (\partial_\xi^{\beta}\partial_\eta^{\gamma}\widetilde{\tau}_{j,\nu})(\xi,\eta)}.
\end{multline*}

Define
$\widetilde{F}(\xi,\eta)=\frac{\xi_\nu+2\eta_\nu}{(1+\abs{\eta}^2)^{\frac{1}{2}}}$
and
$
\widetilde{G}(\xi,\eta)=\frac{\abs{\xi}^2+2\xi\cdot\eta}{1+\abs{\eta}^2}$; 
then we can prove a formula analogous to \eqref{eq:sigmaFG} for
$\partial_\xi^\beta\partial_\eta^\gamma\widetilde{\tau}_{j,\nu}(\xi,\eta)$
with $j>2n+2$ and $\abs{\beta+\gamma}\le 2n+1.$ It follows that the
functions
$\abs{(\xi,\eta)}^{\abs{\beta_1+\gamma_1}} \partial_{\xi}^{\beta_1}\partial_\eta^{\gamma_1}\widetilde{F}(\xi,\eta)$
and
$\abs{(\xi,\eta)}^{\abs{\beta_2+\gamma_2}}
\widetilde{H}_{\beta_2,\gamma_2}(\xi,\eta)$
are bounded on $\supp(\Phi_2)\setminus \{(0,0)\}$ for any
multi-indices and that $\abs{\widetilde{G}(\xi,\eta)}\le \frac{17}{64}$ for $(\xi,\eta)\in \supp(\Phi_2)\setminus \{(0,0)\}$; therefore, we can argue as before to obtain the
desired result for $\widetilde{\sigma}_{j,\nu}.$
\end{proof}

\section{Kato-Ponce inequalities in Lorentz spaces and Morrey spaces}\label{sec:LM}

The tools used to prove Theorem~\ref{thm:KPweighted} also lead to
fractional Leibniz rules in the settings of weighted Lorentz spaces
(more generally, weighted rearrangement invariant quasi-Banach
function spaces) and Morrey spaces. We state these results and briefly
describe their proofs.

\subsection{Kato-Ponce inequalities in weighted Lorentz spaces}

Given $0<p<\infty,$ $0<q\le \infty$ and a weight $w$ defined on $\rn,$  we denote by $L^{p,q}(w)$ the weighted Lorentz space consisting of complex-valued, measurable functions $f$ defined on $\rn$ such that
\[
\|f\|_{L^{p,q}(w)}=\left(\int_0^\infty (t^{\frac{1}{p}} f_w^*(t))^q\,\frac{dt}{t}\right)^{\frac{1}{q}}<\infty,
\]
where $f^*_w(t)=\inf\{\lambda\ge 0:w_f(\lambda)\le t\}$ with
$w_f(\lambda)=w(\{x\in\rn : \abs{f(x)}>\lambda\})$; the obvious changes apply if $q=\infty.$
 It follows that $L^{p,p}(w)=L^p(w)$ for $0<p< \infty.$ 
 
\begin{theorem}\label{thm:KPweightedLorentz}
Let $1< p,q<\infty$ and $\frac{1}{2}<r<\infty$ be such that $\frac{1}{r}=\frac{1}{p}+\frac{1}{q}$ and $0<a,b,c\le \infty$ be such that $\frac{1}{a}=\frac{1}{b}+\frac{1}{c}.$ If $w\in A_{\min\{p,q\}},$    and  $s>\max\{0,n(\frac{1}{r}-1)\}$ or
$s$ is a non-negative even integer, then
for all $f,g\in \sw$, 
\begin{align}
\norm{D^s(fg)}{L^{r,a}(w)} &\lesssim  \norm{D^sf}{L^{p,b}(w)} \norm{g}{L^{q,c}(w)}+ \norm{f}{L^{p,b}(w)} \norm{D^sg}{L^{q,c}(w)},
  \label{eq:KPweightedLorentz1}\\  
\norm{J^s(fg)}{L^{r,a}(w)} &\lesssim  \norm{J^sf}{L^{p,b}(w)} \norm{g}{L^{q,c}(w)}+ \norm{f}{L^{p,b}(w)} \norm{J^sg}{L^{q,c}(w)}.
  \label{eq:KPweightedLorentz2}
\end{align}
The implicit constants depend on $p$, $q$, $s$ and $[w]_{A_{\min\{p,q\}}}$. Moreover, different pairs of $p, q$  and $b,c$ can be chosen for each term on the right hand sides of \eqref{eq:KPweightedLorentz1} and \eqref{eq:KPweightedLorentz2}.
\end{theorem}

The proof of Theorem~\ref{thm:KPweightedLorentz} will follow from a version of Theorem~\ref{thm:sqweighted} for  weighted Lorentz spaces and the boundedness of bilinear Calder\'on-Zygmund operators in weighted Lorentz spaces. 

\begin{theorem}\label{thm:sqweightedLorentz} 
  Let $\Psi\in \sw$ be such that $\supp(\widehat{\Psi}) \subset \{\xi\in\rn: c_1<\abs{\xi}<c_2\}$
  for some $0<c_1<c_2<\infty$.   Given a sequence
  $\bar{z}=\{z_{k,m}\}_{k\in\ent,m\in\ent^n}\subset \rn,$  define
  $\Psi^{\bar{z}}_{k,m}(x)=2^{kn}\Psi(2^{k}(x+z_{k,m}))$ for
  $x\in\rn,$ $m\in\ent^n$ and $k\in\ent.$   Then for every
  $1<p<\infty$, $1\le a\le \infty,$   $w\in A_p,$ and $f\in L^p(w)\cap L^{p,a}(w)$, 
\[
\bigg\|\bigg(\sum_{k\in\ent }\abs{\Psi^{\bar{z}}_{k,m}\ast
  f}^2\bigg)^{\frac{1}{2}}\bigg\|_{L^{p,a}(w)}\lesssim \norm{f}{L^{p,a}(w)},  
\]
where the implicit constants depend on $\Psi$ and $[w]_{A_p}$ but are
independent of $m$ and $\bar{z}.$
\end{theorem}

\begin{proof} This is due to  Theorem~\ref{thm:sqweighted} along with  \cite[Thm 4.10 and comments on p. 70]{MR2797562}, applied to the family of pairs   $\bigg(\bigg(\sum_{k\in\ent }\abs{\Psi^{\bar{z}}_{k,m}\ast f}^2\bigg)^{\frac{1}{2}}, f\bigg),$  $f\in L^p(w)\cap L^{p,a}(w).$ 
\end{proof}

\begin{thm}[See Corollary 6.11 and comments on page 311 in \cite{MR2231047}]\label{thm:CZLorentz}
Let  $1<p,q <\infty$ and $\frac{1}{2}<r<\infty$ be such that $\frac{1}{r}=\frac{1}{p}+\frac{1}{q}$ and $0<a,b,c\le \infty$ be such that $\frac{1}{a}=\frac{1}{b}+\frac{1}{c}.$ If $T$ is a bilinear Calder\'on-Zygmund operator and $w\in A_{\min\{p,q\}}$, then $T$ is bounded from $L^{p,b}(w)\times L^{q,c}(w)$ into $L^{r,a}(w).$
\end{thm}

\begin{proof}[Proof of Theorem~\ref{thm:KPweightedLorentz}] The proof
  of \eqref{eq:KPweightedLorentz1} and \eqref{eq:KPweightedLorentz2}
  now proceeds as in the case of \eqref{eq:KPweighted1} and
  \eqref{eq:KPweighted2}. Indeed, for $j=1,2,3,$ let $T_{j,s}$ and
  $\tilde{T}_{j,s}$ be as in the proof of
  Theorem~\ref{thm:KPweighted}. Theorem~\ref{thm:CZLorentz} gives the
  needed estimates for the operators $T_{j,s}$ and $\tilde{T}_{j,s}$
  for $j=1,2.$ The desired control for the operators $T_{3,s}$ and
  $\tilde{T}_{3,s}$ follows the same ideas as those used for
  \eqref{eq:KPweighted1} and \eqref{eq:KPweighted2}.  Recall that
  $\norm{\cdot}{L^{p,a}(w)}^t$ is comparable to a subadditive quantity
  for $0<t\le 1$ satisfying $t\le a$ and $t<p$ (see \cite[p. 258,
  (2.2)]{MR0223874}) and weighted Lorentz spaces satisfy a
  H\"older-type inequality (see \cite[Thm 4.5]{MR0223874}); then apply
  Theorem~\ref{thm:sqweightedLorentz}.
\end{proof}

\begin{remark} Theorem 4.10 in \cite{MR2797562} and Corollary 6.11 in
  \cite{MR2231047} are also true for weighted versions of
  rearrangement invariant Banach and quasi-Banach function spaces,
  respectively, satisfying certain conditions. An argument like the
  one above for Lorentz spaces  leads to some variants of
  \eqref{eq:KPweightedLorentz1} and \eqref{eq:KPweightedLorentz2} in
  this more general context, and in particular to Orlicz space estimates.
\end{remark}

\subsection{Kato-Ponce inequalities in Morrey spaces}

Given $0\le \kappa \le n$ and $0< p<\infty,$ the Morrey space $\mathcal{L}^{p,\kappa}$ is defined as the class of functions $f\in L^p_{\text{loc}}(\rn)$ such that 
\[
\norm{f}{\mathcal{L}^{p,\kappa}}=\sup_{B}\left(\frac{1}{\abs{B}^{1-\kappa/n}}\int_B\abs{f(x)}^p\,dx\right)^{\frac{1}{p}}<\infty,
\]
where the supremum is taken over all balls $B \subset \rn$. It  is a Banach space, which coincides with $L^p$ for $\kappa=n$
and with $L^\infty$ for $\kappa = 0.$  

\begin{theorem} \label{thm:KPMorrey}
Let $1< p,q<\infty$ and $\frac{1}{2}<r<\infty$ be such that $\frac{1}{r}=\frac{1}{p}+\frac{1}{q}$ and $0<\kappa\le n.$ If  $s>\max\{0,n(\frac{1}{r}-1)\}$ or
$s$ is a non-negative even integer, then
for all $f,g\in \sw$, 
\begin{align}
\norm{D^s(fg)}{{\mathcal L}^{r,\kappa}} &\lesssim  \norm{D^sf}{{\mathcal L}^{p,\kappa}} \norm{g}{\mathcal{L}^{q,\kappa}}+ \norm{f}{\mathcal{L}^{p,\kappa}} \norm{D^sg}{\mathcal{L}^{q,\kappa}},
  \label{eq:KPMorrey1}\\  
\norm{J^s(fg)}{{\mathcal L}^{r,\kappa}} &\lesssim  \norm{J^sf}{\mathcal{L}^{p,\kappa}} \norm{g}{{\mathcal L}^{q,\kappa}}+ \norm{f}{{\mathcal L}^{p,\kappa}} \norm{J^sg}{\mathcal{L}^{q,\kappa}}.
  \label{eq:KPMorrey2}
  \end{align}
  Different triplets of $p,q,\kappa$  can be chosen on the righthand side of \eqref{eq:KPMorrey1} and $\eqref{eq:KPMorrey2}.$
\end{theorem}

As above, the proof of Theorem~\ref{thm:KPMorrey} follows from a
version of Theorem~\ref{thm:sqweighted}  for Morrey spaces and the boundedness of bilinear Calder\'on-Zygmund operators in Morrey spaces. We finish this section with the statements of  such results.

\begin{theorem}\label{thm:sqMorrey}
  Let $\Psi\in \sw$ be such that $\supp(\widehat{\Psi}) \subset \{\xi\in\rn: c_1<\abs{\xi}<c_2\}$
  for some $0<c_1<c_2<\infty$.   Given a sequence
  $\bar{z}=\{z_{k,m}\}_{k\in\ent,m\in\ent^n}\subset \rn,$  define
  $\Psi^{\bar{z}}_{k,m}(x)=2^{kn}\Psi(2^{k}(x+z_{k,m}))$ for
  $x\in\rn,$ $m\in\ent^n$ and $k\in\ent.$   If $0<\kappa\le n,$ then for all $f\in\sw$
\[
\bigg\|\bigg(\sum_{k\in\ent }\abs{\Psi^{\bar{z}}_{k,m}\ast
  f}^2\bigg)^{\frac{1}{2}}\bigg\|_{\mathcal{L}^{p,\kappa}}\lesssim \norm{f}{\mathcal{L}^{p,\kappa}},  
\]
where the implicit constants  are independent of $m$ and $\bar{z}.$
\end{theorem}

Theorem~\ref{thm:sqMorrey} is a direct consequence of  Theorem~\ref{thm:sqweighted} and the following result.

\begin{prop}[see Theorem 6.7 in \cite{JDSummer2005}] Let $\F$ be a family of ordered pairs of non-negative measurable functions defined on $\rn$ and  $1\le p <\infty.$ Suppose that 
 for every $w\in A_1,$  
 \[
 \int_{\rn}\abs{f(x)}^pw(x)\,dx\lesssim \int_{\rn}\abs{g(x)}^pw(x)\,dx
 \]
for all $(f,g)\in \F$, where the implicit constants depend only on $p$ and $[w]_{A_1}.$ If $0 < \kappa\le  n,$ then
\[
\norm{f}{\mathcal{L}^{p,\kappa}}\lesssim \norm{g}{\mathcal{L}^{p,\kappa}} 
\]
for all $(f,g)\in \F.$
\end{prop}

\begin{thm}[Particular case of Theorem 1 in \cite{MR3132676}]\label{thm:CZMorrey}
 Let  $1<p,q <\infty$ and $\frac{1}{2}<r<\infty$ be such that $\frac{1}{r}=\frac{1}{p}+\frac{1}{q}$ and $0<\kappa\le n.$  If $T$ is a bilinear Calder\'on-Zygmund operator,  then $T$ is bounded from $\mathcal{L}^{p,\kappa}\times \mathcal{L}^{q,\kappa}$ into $\mathcal{L}^{r,\kappa}.$
\end{thm}

\def\ocirc#1{\ifmmode\setbox0=\hbox{$#1$}\dimen0=\ht0 \advance\dimen0
  by1pt\rlap{\hbox to\wd0{\hss\raise\dimen0
  \hbox{\hskip.2em$\scriptscriptstyle\circ$}\hss}}#1\else {\accent"17 #1}\fi}

\end{document}